\documentclass[11pt,a4paper]{article}
\usepackage[OT1]{fontenc}
\usepackage[english]{babel}
\usepackage[utf8x]{inputenc}

\usepackage{pgfplots}

\usepackage[warn]{textcomp}

\usepackage{amsmath,amssymb,amsfonts,mathrsfs}
\usepackage[amsmath,thmmarks]{ntheorem}

\usepackage{enumerate}
\usepackage{graphicx}
\usepackage{nicefrac}
\usepackage[normalem]{ulem}
\usepackage{xcolor}
\definecolor{bulgarianrose}{rgb}{0.28, 0.02, 0.03}
\definecolor{chocolate}{rgb}{0.48, 0.25, 0.0}
\definecolor{berry}{rgb}{0.78, 0.31, 0.39}

\usepackage{ulem}

\usepackage{tikz}
\usetikzlibrary{matrix}

\usetikzlibrary{decorations.pathreplacing}

\allowdisplaybreaks

\usetikzlibrary{arrows,positioning} 
\tikzset{
    >=stealth',
    punkt/.style={
           rectangle,
           rounded corners,
           draw=black, very thick,
           text width=6.5em,
           minimum height=2em,
           text centered},
    pil/.style={->,shorten <=0pt,shorten >=6pt,}
}

\usepackage[numbers,square]{natbib}
\usepackage{extarrows}

\theoremprework{\begin{minipage}{\textwidth}}
	\theorempostwork{\end{minipage}}

\newtheorem{theoremw}{Theorem}[section]

\newtheorem{lemma}[theoremw]{Lemma}
\newtheorem{proposition}[theoremw]{Proposition}
\newtheorem{def+prop}[theoremw]{Definition+Proposition}
\newtheorem{theorem}[theoremw]{Theorem}
\newtheorem{conjecture}[theoremw]{Conjecture}

\theorembodyfont{\normalfont}

\newtheorem{definition}[theoremw]{Definition}
\newtheorem{remark}[theoremw]{Remark}

\theoremstyle{nonumberplain}
\theorembodyfont{\normalfont}
\theoremsymbol{\ensuremath{\square}}
\newtheorem{proof}{Proof}

\newcommand{\N}{\mathbb{N}}

\newcommand{\Z}{\mathbb{Z}}

\newcommand{\Prob}{\mathbb{P}}

\newcommand{\Pa}{\mathsf{P}}
\newcommand{\E}{\mathbb{E}}

\newcommand{\GW}{\mathsf{GW}}
\newcommand{\GWm}{\mathsf{GW^{mult}}}
\newcommand{\AGW}{\mathsf{AGW}}
\newcommand{\MR}{\mathsf{MoveRoot}}
\newcommand{\SRW}{\mathsf{SRW}}

\newcommand{\BMC}{\mathrm{BMC}}
\newcommand{\FM}{\mathrm{FM}}
\newcommand{\connect}{\bullet \hspace{-4pt}- \hspace{-4pt}\bullet}

\newcommand{\fmm}{\mathsf{FM}}

\newcommand{\ST}{\mathsf{ST}}
\newcommand{\PER}{\mathsf{PER}}
\newcommand{\GWbb}{{\GW}_{\mathsf{bg}}}

\newcommand{\tb}{\mathsf{b}}
\newcommand{\tg}{\mathsf{g}}
\newcommand{\tbr}{\mathsf{b_r}}

\newcommand{\les}{\mathsf{es}}
\newcommand{\lbs}{\mathsf{bs}}
\newcommand{\ls}{\mathsf{s}}
\newcommand{\lno}{\mathsf{n}}

\newcommand{\pa}{\mathsf{p}}

\newcommand{\f}{\mathsf{f}}

\newcommand{\T}{\mathbb{T}}
\newcommand{\Tr}{\mathsf{T}}

\title{On transience of frogs on Galton--Watson trees}
\author{Sebastian M\"uller, Gundelinde Maria Wiegel }

\newcommand*{\SwitchToOpenAny}{\csname @openrightfalse\endcsname}
\newcommand*{\SwitchToOpenRight}{\csname @openrighttrue\endcsname}

\begin{document}
\maketitle

\begin{abstract}
We consider a interacting particle system, known as the \textit{frog model}, on infinite Galton--Watson trees allowing offspring 0 and 1. The system starts with one awake particle (frog) at the root of the tree and a random number of sleeping particles at the other vertices. Awake frogs move according to simple random walk on the tree and as soon as they encounter sleeping frogs, those will wake up and move independently according to simple random walk. The frog model is called transient if there are almost surely only finitely many particles returning to the root. In this paper we prove a 0--1-law for transience of the frog model and show the existence of a transient phase for certain classes of Galton--Watson trees.


\end{abstract}

\section{Introduction} \label{introduction}
The frog model $\FM(X,\eta,\Pa)$ is a random interacting particle system, consisting of three parts: a graph $X$ with a dedicated root, an i.i.d.\ configuration of sleeping frogs on each vertex according to a distribution $\nu$, and a path measure $\Pa$ describing the movement of the frogs.
We identify the graph with its vertex set and assume that the mean number of frogs $\bar{\eta}$ is finite. 

The model starts by definition with one awake frog at the root $o$ of the graph $X$ and i.i.d.\ sleeping particles according to $\eta$ at the other vertices. The awake frogs move independently on the graph with respect to $\Pa$. When a vertex with sleeping frogs is visited for the first time, the sleeping frogs at this vertex wake up and start to move according to $\Pa$ independently of the other frogs. The different frog models can vary in the underlying graph, the initial distribution of the sleeping frogs (deterministic or random) and the path measure of the awake frogs. Unless it is not specified otherwise, we assume that the frogs move according to simple random walk. We write from now on $\SRW$ instead of $\Pa$, and shorten the notation from $\FM(X,\eta,\SRW)$ to $\FM(X,\eta)$. More precisely,  for $v,w\in X$ we consider the transition probabilities $p(v,w)=\frac{1}{deg(v)}$ if $v$ is a neighbour of $w$, and $0$, otherwise.

In 1999 the frog model was originally introduced as ``egg model" in \cite{telcs} and later on Rick Durrett established the name ``frog model''. One main point of interest since its introduction was studying the recurrence and transience of the frog model. 
Let  $\fmm:=\fmm_{X}:=\SRW \times \eta$ denote the probability measure on paths of all frogs (following the dynamics of a  SRW) given by choosing an independent and identically distributed initial frog configuration according to $\eta$ on the graph  $X$.
Moreover, we define the random variable
\[\nu:= \# \, \text{of visits  to the root} ,\]
which is the aggregated number of visits to the root in the frog model.  We define recurrence and transience in the following way:
\begin{definition} \label{def:trans+rec}
Let $X$ be a graph with a dedicated root $o$. The frog model $ \FM(  X,\eta)$ is called {transient} if 
\[\fmm[\nu <\infty]=1\, ,\]
that is, there are $\fmm$-almost surely only finitely many visits to the root. Otherwise the frog model is called {recurrent}.
\end{definition}

Studying transience and recurrence of the frog model is only interesting when the single random walk is transient. The first result concerning the question about recurrence was in the aforementioned article \cite{telcs}, where Telcs and Wormald  showed that $\FM(\Z^d,\delta_1,\SRW)$ is recurrent for all $d\in \N$. Later Gantert and Schmidt showed conditions for recurrence for the frog model with drift on the integers in \cite{gantertschmidt}. This was generalized to higher dimensions and a drift in the direction of one axis by D\"obler and Pfeifroth  \cite{doblerpfeiffroth} and  D\"obler {\it et al.}~\cite{gantertweidner2018}.

In 2002,  Alves, Machado and Popov \cite{popovphase} studied the frog model on trees with the modification, that the frogs can die with a certain probability $p$ in each step. Let $p_c$ denote the smallest $p$ such that the frog model survives with positive probability. In \cite{popovphase} they are prove in which cases there exists a phase transition, that is $0<p_c<1$, on homogeneous trees and integer lattices. Moreover, they have proven phase transitions between transience and recurrence with respect to the survival probability. In 2005 there was the first improvement of the upper bound of $p_c$ by Lebensztayn, Machado and Popov \cite{Lebensztayn2005}. Recently, Lebensztayn and Utria improved the result again in \cite{Lebensztayn2019} and proved an upper bound for $p_c$ on biregular trees in \citep{Lebensztaynbiregular}. Another modification of the frog model was considered by Deijfen, Hirscher and Lopes in  \cite{deijfenhirscher} and by Deijfen and Rosengren in \cite{deijfenrosengren}. These two papers work on a two-type frog model performing lazy random walk. They show that two populations of frogs on $\Z_d$ can coexist under certain conditions on the path measure of the frogs. Moreover, the coexistence of the frog model does not depend on the shape of the initially activated sets and their frog configuration.

The question if $\FM(T_{d+1},\delta_1,\SRW)$ on the homogeneous tree, or $d+1$-regular tree, $T_{d+1}$  is recurrent or transient remained open for quite some time. In 2017 Hoffmann, Johnson and Junge could show in \cite{onefrogpersite}, that $\FM(T_{d+1},\delta_1,\SRW)$ is recurrent for $d=2$ and transient for $d\geq 5$. This result was extended by Rosenberg \cite{rosenberg} showing that the alternating tree  $T_{3,2}$  with offspring $3$ and $2$ is recurrent.  Studying the frog model on trees was continued by modifying the frog configuration $(\eta(x))_{x\in T_{d+1}}$ to $pois(\mu)$-distributed frogs. Hoffmann, Johnson and Junge proved in \cite{poissonfrogsone} the existence of a critical parameter $\mu_c$, bounded by $C d < \mu_c(d) < C' d \log d$ with $C,C'$ constants, such that $\FM(T_{d+1},pois(\mu),\SRW)$ is recurrent for $\mu > \mu_c$, and transient for $\mu< \mu_c$. Johnson and Junge improved the bounds to $0.24d \leq \mu_c(d) \leq 2.28d$ for sufficiently large $d$ in \cite{poissonfrogstwo}. 

The subtlety of the question of recurrence and transience is also reflected in a  result by Johnson and Rolla \cite{JoRo:18}. In fact,  transience and recurrence are sensitive not just to the expectation of the frogs  but to the entire distribution of the frogs. This is in contrast to closely related models like branching random walk and activated random walk.

Very recently, Michelen and Rosenberg proved in \cite{RosenbergMichelen} the existence of a phase transition between transience and recurrence on Galton--Watson trees. This was done for trees of at least offspring two. In this paper we want to answer an open question which appeared in \cite{RosenbergMichelen}, and extend their result. We will prove the existence of a transient phase for supercritical Galton--Watson trees with bounded offspring but also allowing offspring 0 and 1.  As in the references above we assume that the initial distribution is random according to a distribution $\eta$ with finite first moment. We start with showing a 0--1-law for transience. 
\begin{theorem} \label{01law}
Let $\GW$ be the measure of a Galton--Watson tree and $\Tr$ a realization. Then, 
\[\GW[ \ \FM(  \Tr,\eta) \, \text{is transient} \,| \, \Tr \text{ is infinite} ]\in \{0,1\}.\]
\end{theorem}
Michelen and Rosenberg recently proved a stronger 0--1-law for recurrence and transience in \cite{RosenbergMichelen}. In fact, in our paper, recurrence is defined as not being transient, meaning that infinitely many particles return to the root with positive probability. The 0--1-law of  \cite{RosenbergMichelen} treats the stronger definition of recurrence; the process is recurrent if almost surely infinitely many particles return to the root.

We learned about their proof after writing our first version. While both proofs rely on the stationarity of the  augmented Galton-Watson measure, our proof differs in the connection between the ordinary Galton--Watson measure and the augmented Galton--Watson measure. 
In \cite{zerkosfrogs} Kosygina and Zerner proved a 0--1-law for transience and recurrence of the frog model on quasi-transitive graphs. 

The main result of the paper is the existence of a transient phase while allowing offspring 0 and 1:
\begin{theorem} \label{main}
Let  $\GW$  be a Galton--Watson measure defined by $(p_{i})_{i\geq 0}$. We assume that  $d_{max}=\max\{i : p_{i}>0\}<\infty$ and set $d_{min}:=\min\{i\geq 2: p_{i}>0\}$. Then, for any choice of $p_{0}$ and $p_{1}$ there exist some constants $c_{d}=c_{d}(p_{0},p_{1})$ and $c_{\eta}=(p_{0},p_{1}, d_{max})$ such that for $d_{min}\geq c_{d}$ the frog model $\FM(\Tr,\eta,\SRW)$ is transient $\GW$-almost surely (conditioned on $\Tr$ to be infinite) if $\bar{\eta} < c_\eta$.
\end{theorem}

We recall that $\bar{\eta}$ is the expected value of the number of sleeping frogs at each vertex. The assumption of finite maximum offspring is needed to control the possible number of attached bushes in a Galton--Watson tree allowing offspring $0$. We want to note that the proof in the case with stretches and no bushes does not need this assumption.
The proof of Theorem \ref{main} gives bounds on the constants. These bounds can certainly be improved in refining the involved estimates, see Figure \ref{fig:dMsearch} for some explicit values.
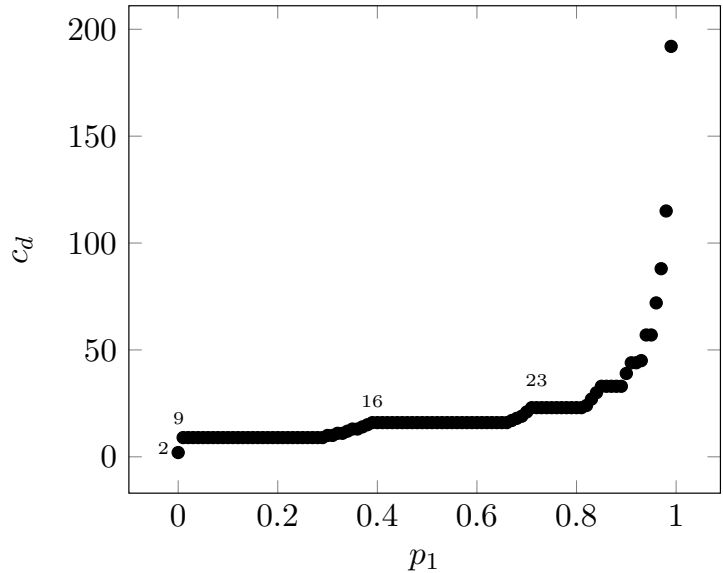
\begin{figure}[th]
 \centering
  \resizebox{0.65\textwidth}{!}{
\begin{tikzpicture}
\begin{axis}[xlabel=$p_1$, ylabel=$c_{d}$]
\addplot[only marks ] table [col sep=comma] {searchdn_newdmin.csv};
\node[] at (-30,2) {\tiny{$2$}};
\node[] at (1,16) {\tiny{$9$}};
\node[] at (390,24) {\tiny{$16$}};
\node[] at (720,34) {\tiny{$23$}};
\end{axis}
\end{tikzpicture}
    }
\caption{The minimal $d_{\min}=c_{d}$ for each $p_1$, such that there exists a transient phase of the frog model; with $p_{0}=0$. The mesh size for $p_{1}$ is $0.01$.}
 \label{fig:dMsearch}
\end{figure}
We believe that a different approach or a new perspective is needed to  prove the following conjecture. 

\begin{conjecture}
Let  $\GW$  be a Galton--Watson measure defined by $(p_{i})_{i\geq 0}$ with mean offspring $\sum_{i} i {p_{i}}>1$. Then, there exists some constant $c=c(\GW)>0$ such that the frog model $\FM(\Tr,\eta,\SRW)$ is transient for $\GW$-almost all infinite realization $\Tr$ if $\bar{\eta} < c$.
\end{conjecture}

For proving Thorem \ref{main} we compare the frog model with a branching Markov chain (BMC). In contrast to the frog model, the particles in the BMC branch at every vertex, regardless if they visited the vertex already or not. Therefore,  there are more particles in the BMC than in the frog model and we can couple the two models. In this way, transience of the BMC implies transience of the frog model. The same kind of approach was already used for example in the proofs of transience 
in \cite{poissonfrogsone} and \cite{poissonfrogstwo}. 

While on homogeneous trees the existence of a transient branching Markov chain is guaranteed, this is no longer true in general for Galton--Watson trees. Namely, allowing  the particle to have  $0$ and $1$ offspring creates stretches and finite bushes in the family tree. Such trees have a spectral radius equal to $1$ and therefore the branching Markov chain is always recurrent on such trees, see \cite{mullerganterttransience}. To tackle this problem, we first modify the  Galton--Watson trees and then adapt the branching Markov chain to get a dominating process. Firstly, we start with dealing with arbitrary long stretches. This turns out to be more difficult than expected, since a direct coupling of the frog model and the branching Markov chain is not possible. For this reason we compare the expected number of returns to the root of the frog model with the expected number of returns of annother, appropriate branching Markov chain. Next, we treat the case of appearing bushes and possible stretches. This part is essentially a rather straightforward generalization of the first part. The main idea is to control the bushes and the ``backbone'' (the tree without bushes) separately. The backbone is essentially a Galton--Watson tree with stretches and the bushes just increase the number of frogs per vertex. 

The paper is structured in the following way. In Section \ref{sectionGW} we give an introduction to Galton--Watson trees  and state some useful structural results. Then, we recall the definition of a branching Markov chain together with the above stated transience criterion in Section \ref{sectionBMC}. The 0--1-law is proved in Section \ref{sec:0-1}. The proof of Theorem \ref{main} will be split in three parts. In Subsection \ref{subsec:nobuschnostretch} we treat the case of no bushes and no stretches ($p_1=0,p_0=0$), in Subsection \ref{subsec:nobushesbutstretches}  the case when there are no bushes, but stretches ($p_1>0,p_0=0$), and in Subsection \ref{subsec:bushesandstretches} the case when we have bushes and possibly also stretches ($p_0>0$).

\section{Galton--Watson trees}\label{sectionGW}
The {Galton--Watson tree} (GW-tree) is the family tree of a {Galton--Watson process}. This latter process starts with one particle at time $0$ and at each discrete time step every particle  generates new particles independently of the previous history and the other particles of the same generation. More formally, let $Y$ be a non-negative integer valued random variable with $p_{k}:=\Prob[Y=k]$ for each $k\in \N$ and let $m:=\sum_{k\geq 0} k\, p_k$ be the mean of $Y$. Moreover, let $Y_i^{(n)}$, $i,n\in \N$, be independent and identically distributed random variables with the same distribution as $Y$. Then, the Galton--Watson process is defined by $Z_0 :=1$ and 
\[Z_n:=\sum_{i=1}^{Z_{n-1}} Y_{i}^{(n)}\, \] 
for $n\geq 1$.  The random variable $Z_n$ represents the number of particles in the  $n$-th generation. 
A GW-process with $p_0>0$ will survive with positive probability, that is $\Prob[Z_n >0 \,\ \mbox{for all} \,\ n>0]>0$, if and only if $m>1$.  We introduce $\T$ as the random variable for the family tree of the GW-process and its corresponding measure by $\GW$. Moreover, we denote by $\Tr:=\T(\omega)$ a fixed realization of $\T$. 
In the remaining paper we only consider GW-trees with bounded number of  offspring: There exists a $d_{max} \in \N$, such that $\sum_{k=0}^{d_{max}} p_k=1$.  
For a more detailed introduction to GW-processes and trees we refer to Chapter 5 in \cite{baumbibel}. 

In the case where $p_0>0$  the  GW-tree  contains a.s.~finite bushes. We will distinguish between two types of vertices.
\begin{definition}
We call a vertex $v \in \Tr$ of {type $\tg$} if it lies on an infinite geodesic starting from the root. Otherwise we call vertex $v$ of {type $\tb$}.
\end{definition}
If a vertex of type $\tb$ is the child of a type $\tg$ vertex we call it of {type $\tbr$} and speak of it as the root of the finite bush that consists of its descendants.

We set
\[f(r):= \E\left[r^Z\right]=\sum_{k\geq0} r^k p_k\]
 as the generating function of the GW-process and $q$ the smallest solution of $f(r)=r$.
 
 Let us consider the case  where $p_0>0$ and describe the distribution of the  tree $\T$ conditioned to be infinite.  We start with a tree $\T^*$ generated according to the generating function 
\[f^*(s):= \frac{f(q+(1-q)s)-q}{1-q}\, .\]
This tree will serve as the backbone of $\bar{\T}$ and looks like a  supercritical GW-tree without leaves. All vertices in this tree are of type $\tg$. To each of the vertices of $\T^*$ we attach a random number of independent copies of a sub-critical GW-tree generated according to
\[\tilde{f}(s):=\frac{f(qs)}{q}\, .\]
These are finite bushes consisting of vertices of type $\tb$.
The resulting tree $\bar{\T}$ has the same law as $\T$, conditioned on nonextiction and is a multitype GW-tree with vertices of type $\tb$ and $\tg$. We denote the measure generating $\bar{\T}$ by $\GWm$, e.g.~see Proposition 5.28 in \citep{baumbibel}.

Let $(Z^{sub}_n)_{n\geq 0}$ denote the subcritical Galton--Watson process with probability generating function $\tilde{f}$ and $\T^{sub}$ its family tree. We know that $\E[Z^{sub}_1]<1$ and moreover it holds, e.g.,~Theorem 2.6.1 in \cite{jagers} that 
\begin{align}
\lim_{n\rightarrow \infty} \frac{\Prob[Z^{sub}_n>0]}{\E[Z^{sub}_1]^n}=c \quad \text{and} \quad \E\left[\lvert\T^{sub}\lvert\right]<\infty\, . \label{populationsubcritical}
\end{align}

Now, if we assume that $p_1>0$ the resulting GW-tree may contain  arbitrary long stretches. We want to show that this tree generated by $\GW$ is equivalent to a tree generated in  three steps where firstly the tree without stretches is generated, secondly the location of the stretch is determined and thirdly the stretches are inserted. Therefore we define a new GW-measure using the modified offspring distribution
\[\hat{p}_k:=\frac{p_k}{1-p_1}\]
for $k=0,2,\ldots,N$ and let ${\GW}_{\mathsf{bg}}$ be the measure generating a tree with this distribution. Let us denote a tree generated by $\GWbb$ with $\T_{\tb\tg}$. In the next step every vertex will be independently labeled with $\lbs$ with probability $p_1$, which denotes the starting point of a stretch. If such a vertex has no offspring we attach one vertex, otherwise insert a vertex with offspring one in between the vertex and its descendants, see Figure \ref{figure:Baumstretch}. We write for such a tree $\T_{\pa\times\tb\tg}$. In the next step, the length of the stretch attached to a vertex with label $\lbs$ will be distributed according to $L$ were $L$ is geometrically distributed $geo(p_1)$ and we obtain a tree $\T_{\ls\times\pa\times\tb\tg}$.  This yields  $1+geo(p_1)$ distributed vertices with offspring one in a row. 
\begin{figure}[h]
 \centering
 \begin{minipage}[t]{0.24\textwidth}
 \vspace*{0pt}
\includegraphics[width=1\textwidth]{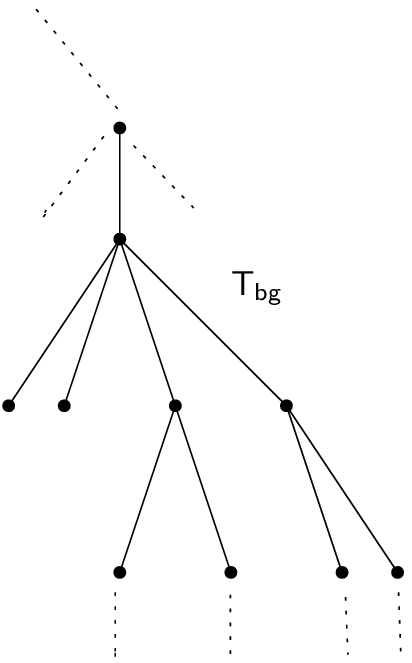}
\end{minipage}
\hspace{6mm}
\begin{minipage}[t]{0.24\textwidth}
\vspace*{0pt}
\includegraphics[width=1\textwidth]{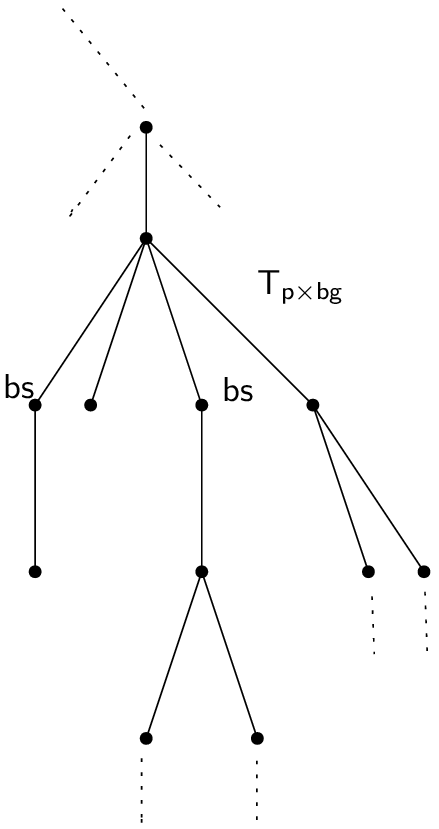}
\end{minipage}
\hspace{6mm}
\begin{minipage}[t]{0.24\textwidth}
\vspace*{0pt}
\includegraphics[width=1\textwidth]{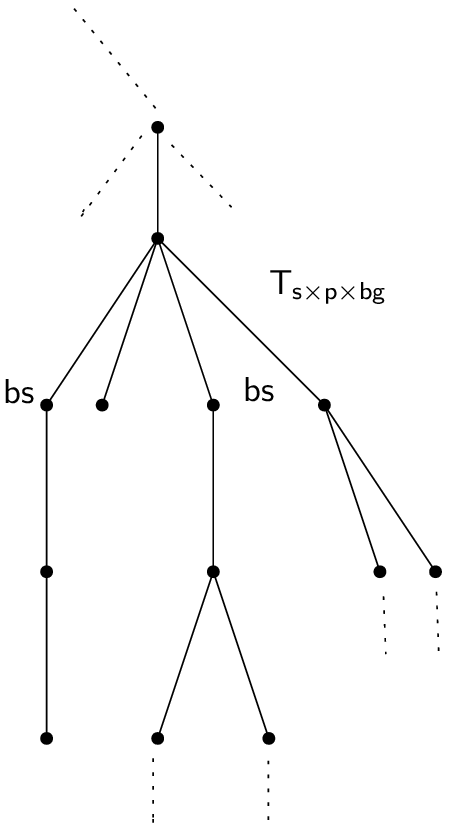}
\end{minipage}
\caption{Realizations of ${\T}_{\ls\times \pa\times \tb\tg}$ by $\ST\times \PER\times \GWbb$ step by step.}
 \label{figure:Baumstretch}
\end{figure}
The length of the stretches will be determined for each stretch starting point independently and identically distributed. We will call this measure of selecting a stretch point $\PER$ and the one of choosing the length of the stretch by $\ST$. We denote by  $\T_{\ls\times\pa\times\tb\tg}$ the tree constructed in the three steps according to $\ST\times \PER\times \GWbb$. 
The resulting tree has the same distribution as the tree constructed as follows: We start with a root and proceed inductively. Every new vertex
\begin{itemize}
\item has $0$ descendants (new vertices)  with probability  $\hat{p}_0(1-p_1)=p_0$,
\item has $k\geq2$ descendants with probability $\hat{p}_k (1-p_1)=p_k$,
\item is the starting point of a stretch with length $\ell+1$ and the end of the stretch has  $k =0,2,\ldots,N$ descendants  with probability $ p_1 p_1^{\ell} (1-p_1) \hat{p}_k=p_1 p_1^{l} {p_k}$.
\end{itemize}
Moreover,  it holds for any finite tree $\Tr$, that
$ \GW(\Tr')=\ST\times \PER\times \GWbb (\Tr'')\, $ where $\Tr'$ and $\Tr''$ are GW-trees starting with $\Tr$.
Therefore the two measures $ \GW$ and $\ST\times \PER\times \GWbb$ are equivalent on the space of all rooted locally finite trees.

\section{Branching Markov chain} \label{sectionBMC}
One method for proving transience of the frog model relies on the comparison of the frog model to a {branching Markov chain (BMC)}. 

A BMC is a cloud of particles that move on an underlying graph $G$ in discrete time. The process starts with one particle in the root $o$ of the graph. Each particle splits into offspring particles at each time step, which then move one step according to a Markov chain on $G$. Particles branch and move independently of the other particles and the history of the process. 
We denote $\mu(v)=(\mu_{k}(v))_{k\in \N}$ for the offspring distribution in a vertex $v$; $\mu_{k}(v)$ is the probability that a particle in $v$ splits into $k$ particles. 

A BMC  can also be seen as a labeled Galton--Watson process or tree-indexed Markov chain, \cite{BePe:94a}, where the labels correspond to the particles' position. 
In our setting the particles will move on a tree $T$ according to   the transition operator $P=(p(v,w))_{v,w\in T}$ of a simple random walk (SRW). We note $p^{(n)}(v,w)$ for the $n$-step probabilities.
If $T$ is connected, the SRW is irreducible and the {spectral radius}
\[ \rho(T):= \rho(P):=\limsup_{n\rightarrow \infty}\left( p^{(n)}(v,w)\right)^{\frac{1}{n}}, v,w\in T\]
is well-defined and takes values in $(0,1]$. 

We add the branching mechanism that in every vertex $v\in T$ a particle arriving at $v$ branches according to a branching distribution $\mu(v)$; i.e.~each $\mu(v)$ is a measure on $\N$. We denote by $\mu$ the whole sequence $(\mu(v))_{v\in T}$. The expected value of each branching distribution is
\[\bar{\mu}(v)=\sum_{k\in \N} k \mu_k(v)\]
for all $v \in T$ where $\mu_k(v)$ is the probability that a particle jumping to $v$ branches into $k \in \N$ particles. We set  $\BMC(T,P,\mu)$ for this branching Markov chain.

Similarly to the frog model, the $\BMC$ is called {transient} if the root will be visited almost surely only by finitely many particles. Otherwise it is called {recurrent}. A particular case of the transience criterion for BMC given in \citep{mullerganterttransience} is the following.
\begin{theorem} Let $T$ be a locally finite tree and $P$ the transition of the SRW on $T$. We assume that all branching distributions $\mu(v)$, $v \in T$, have the same mean   $\bar{\mu}>1$. Then the $\BMC(T,P,\mu)$ is transient if and only if 
\[\bar{\mu}\leq \frac{1}{\rho(T)}\, .\] \label{transiencecriterium}
\end{theorem}

\begin{remark}
A  BMC with $\bar \mu=1+\bar \eta$ is a natural candidate to bound the frog model with sleeping frogs distributed according to $\nu$. Let us consider the open problem of transience of the ``one frog per vertex frog model''  on the $4$-ary tree $\Tr_{5}$,  see \cite{onefrogpersite}. The spectral radius of the SRW on $\Tr_{5}$ is $\rho(\Tr_{5})=4/5.$ In the case of one frog per vertex we have that $\bar \mu=2$ and the criterion can not be applied to obtain transience. However, if the probability of having one sleeping frog is less than $1/4$ and zero frogs otherwise,  then $\bar \mu \leq 5/4$ and the BMC is transient.
\end{remark}

\section{0--1-law for transience}\label{sec:0-1}
Before proving the existence of a transient phase for the frog model we want to show that the existence of a transient phase does not depend on the specific realization of the GW-tree. In other words, we show that the frog model is either transient for $\GW$-almost all infinite trees or recurrent for $\GW$-almost all infinite trees. 

The proof of  this 0--1-law,  Theorem \ref{01law}, relies on the concept of the environment viewed by the particle. We prove that the events of transience and recurrence are invariant under re-rooting and hence the 0--1-law follows from the ergodicity of the augmented GW-measure.

The {augmented Galton--Watson measure}, denoted by  \textbf{$\AGW$}, is a stationary version of the usual Galton--Watson measure. This measure is defined just like $\GW$ except that the number of children of the root  has the law of $Y + 1$; i.e.~the root has $k + 1$ children with probability $p_{k}$.   The measure $\AGW$ can also be constructed as follows: choose two independent copies $\Tr_{1}$ and $\Tr_{2}$ with roots $o_{1}$ and $o_{2}$ according to $\GW$ and connect the two roots by one edge to obtain the tree $\Tr$ with the root $o_{1}$. We write $\Tr= \Tr_{1} \connect \Tr_{2}$.

We consider the Markov chain on the state space of rooted trees. 
If we change the root of a tree $T$ to a vertex $v\in T$, we denote the new rooted tree by $\MR(T,v)$. We define a Markov chain on the space of rooted trees as:
\begin{align*}
p_{\SRW}((T,v),(T',w))=
\begin{cases} \frac{1}{deg(v)}, \,\ &\text{if} \,\  v\sim w\,\  \text{and}\,\ (T',w)=\MR(T,w), \\
0, \,\ &\text{otherwise}.
\end{cases}
\end{align*}
By Theorem 3.1 and Theorem 8.1 in \cite{peres:ergodic} it holds that this Markov chain with transition probabilities $p_{\SRW}$ and the initial distribution $\AGW$ is {stationary} and {ergodic} conditioned on non-extinction of the Galton--Watson tree. 

\begin{lemma} \label{lemma:invariant}
The events of transience and recurrence of the frog model are invariant under changing the root of the underlying rooted tree $T=(T,o)$, i.e.~$\FM(T, \eta)$ is transient if and only if  $\FM(\MR(T,v), \eta)$ is transient for some (all) $v\in T$.
\end{lemma}
\begin{proof}
As the case of finite trees is trivial we consider an infinite rooted tree $(T,o)$ and let $v\sim o$.  We proof that transience of $(T,o)$ implies transience of $(T,v)$ by assuming the opposite. If $\FM(\MR((T,v), \eta)$ is recurrent, then there exists some $k\in \N$ such that with positive probability infinitely many frogs visit $v$ conditioned on $\eta(o)=k$. In the frog model $\FM((T,o),\eta)$ conditioned on $\eta(v)=k$, the starting frog in $o$ jumps to $v$ with positive probability. Again with positive probability at the second step  all frogs awaken in $v$ jump back to  $o$ while the frog that came from $o$ is assumed to stay in $v$ for one time step. Note that this has no influence on transience or recurrence of the process.
 This recreates the same  initial configuration of $\FM((T,v), \eta)$ conditioned on $\eta(o)=k$  with the difference that more frogs are already woken up. By assumption in this process infinitely many particles visit $v$ with positive probability, and hence, by the Borel--Cantelli Lemma,  also $o$ is visited infinitely many times with positive probability. A contradiction. The claim for arbitrary $v$ now follows by induction and connectedness of the tree.
\end{proof}

\begin{proof}[Theorem \ref{01law}]
By the ergodicity of the Markov chain with  transition probabilities $p_{\SRW}$ and Lemma \ref{lemma:invariant}, it holds that 
\begin{align*}
&\AGW [\FM(\Tr, \eta)\, \text{ transient} \, \lvert \, \Tr \, \text{infinite} ]\in \{0,1\}\, . 
\end{align*}
We prove first that \[\GW[ \FM(\Tr, \eta)   \, \text{is transient}]>0\] implies  \[\AGW[\FM(\Tr, \eta)  \, \text{ transient}]>0\, .\] 
Let $\Tr_{1}$ be a realization on which the frog model is transient. Then, there exists some ball $\mathcal{B}$ around the root $o_{1}$ such that no frog awaken outside this ball $\mathcal{B}$ will visit the origin $o_{1}$. Let $\Tr_{2}$ be an independent realization according to $\GW$ and let $\Tr:= \Tr_{1} \connect \Tr_{2}$. 

In the frog model on $(\Tr, o_{1})$  the starting frog jumps into $\Tr_{1}$ at time $n=1$ with positive probability. Now, since every frog is transient,
with positive probability all frogs in the set $\mathcal{B}$ that are woken up will never cross the additional edge $(o_{1}, o_{2})$  and we obtain that $\AGW[ \FM(\Tr, \eta)  \, \text{transient}]>0$.  We write $\AGW_{\infty}[\cdot]:= \AGW[\cdot \, \lvert \, \Tr \, \text{infinite}]$ and define $\GW_{\infty}$ similarly. The 0--1-law gives that $\AGW_{\infty}[\FM(\Tr, \eta)  \, \text{transient}]=0$ implies $\GW_{\infty}[\FM(\Tr, \eta)  \, \text{transient}]=0$.

It remains to show that \[\GW_{\infty}[\FM(\Tr, \eta)  \, \text{recurrent}]>0\] implies that \[\AGW_{\infty}[\FM(\Tr, \eta)  \, \text{recurrent}]>0\, .\] Let $\Tr_{1}$ and $\Tr_{2}$ be two recurrent realizations of $\GW_{\infty}$ and let $\Tr:=\Tr_{1} \connect \Tr_{2}$. Each copy $\Tr_{i}, i\in\{1,2\}$, is recurrent with positive probability. Hence, we have to verify that the possibility that frogs can change from one  $\Tr_{i}$ to the other does not change this property. Let us say that every frog originally in $\Tr_{1}$ wears a red T-shirt and every frog in $\Tr_{2}$ wears a blue T-shirt. Now,  every frog that jumps from $o_{1}$ to $o_{2}$ leaves its red T-shirt in a stack in $o_{1}$. In the same way every frog leaving $o_{1}$ to $o_{2}$ leaves its blue T-shirt in a stack in $o_{2}$. A frog arriving from $o_{1}$ to $o_{2}$ takes a blue T-shirt from the stack. If the stack is empty, the frog ``creates'' a new  blue shirt. We proceed similarly for the frogs that arrive in $o_{1}$ coming from $o_{2}$. The frog model $\FM(\Tr, \eta)$ starts with one awoken frog in a red T-shirt in $o_{1}$. Once a frog visits $o_{2}$, the blue frog model  $\FM(\Tr_{2}, \eta)$ is started and a red shirt is left in $o_{1}$. Conditioned on the event that  $\FM(\Tr_{2}, \eta)$ is recurrent a blue frog will eventually jump from $o_{2}$ to $o_{1}$ and put on the red shirt. In this way, every red shirt is finally put on and the distribution of the red frogs in $\FM(\Tr,\eta)$ equals the distribution of the frogs in $\FM(\Tr_{1},\eta)$ with possible additional frogs. In other words, $\FM(\Tr,\eta)$ is recurrent with positive probability.

Finally, we can conclude
\begin{align*}
&\GW [\FM(\Tr, \eta)\, \text{is transient} \, \lvert \, \Tr \, \text{is infinite}]\in \{0,1\}\, . 
\end{align*}
\end{proof}

\section{Transience of the frog model} \label{transiencephase}
\subsection{No bushes, no stretches} \label{subsec:nobuschnostretch}
We start with considering GW-trees $\T$ with $p_0+p_1=0$. By Lemma \ref{lemmaspectralradii} we know that $\rho(\T)<1$ and hence  Theorem \ref{transiencecriterium} guarantees a transient phase for $\BMC$ on such GW-trees $\T$. Coupling the frog model with an appropriate branching Markov chain implies a transient phase for the frog model.

\begin{lemma} \label{couplingstandard}
Consider a Galton--Watson measure $\GW$ with $p_0+p_1=0$ and $m>1$. Then,  for $\GW$-almost all trees $\Tr$ the frog model with $\eta$ distributed number of frogs per vertex is transient if mean $\bar{\eta} \leq  \frac{d+1}{2\sqrt{d}} -1$ where $d:=min\{k:p_k>0\}$.
\end{lemma}

\begin{proof}
The proof relies on the fact that the  $\BMC(\Tr,P,\mu)$, where $\mu(v)$ fulfills $\mu_k(v)=\Prob [ \eta(v) +1=k]$ for each $k \geq 1$ and $v\in \Tr$, stochastically dominates the frog model. 
We use a coupling of the frog model with a $\BMC$ such that at most as many frogs (in the frog model) as particles (in the BMC)  visit the root. More precisely, in both models we start with one frog, respectively particle, at the root and couple them. A particle of the $\BMC$ that is coupled to a frog  $x$ in the frog model is denoted by $x'$. The ``additional'' particles in the BMC, in the meaning that they have no counterpart in the frog model, will move and branch without having any influence on the coupling. 
Let $(\f_v)_{v \in \Tr}$ be a realization of the sleeping frogs. If a first coupled particle arrives at $v$ it branches  into $\f_v+1$ particles. The awakened frogs and newly created particles are coupled. If more than one coupled particle arrives at $v$ for the first time at the same moment, we choose (randomly) one of these, let it have  $\f_v+1$ offspring and couple the resulting particles with the frogs as above. The offspring of the other particles (those that are coupled to the remaining frogs arriving at $v$) are chosen i.i.d.\ according to $\mu(v)$ and one of them (randomly chosen) is coupled to each corresponding frog.
Similarly, if a vertex $v\in \Tr$ will be visited a second time by a frog, no new frogs will wake up but the particle will branch again into a random $\mu(v)$ distributed number of particles and we couple the frog arriving at $v$ with one (randomly chosen) of the particles. In this way every awake frog is coupled with a particle of the $\BMC$. Hence if the $\BMC$ is transient, then also the frog model is transient.
The mean offspring $\bar{\mu}$ of $\BMC(\Tr,P,\mu)$ is  constant
\[\bar{\mu}:=\bar{\mu}(v)=\E[\eta(v)]+1=\bar{\eta} +1\, \]
for any $v\in\Tr$ as $\eta(v)$ are independent and identically distributed.
Using Theorem \ref{transiencecriterium} it follows that the $\BMC$ is transient if and only if 
\begin{align*}
\bar{\eta}+1 \leq \frac{1}{\rho(\Tr)} \, .
\end{align*}
By Lemma \ref{lemmaspectralradii} it holds that $\rho(\Tr)= \rho(T_{d+1})=\frac{2\sqrt{d}}{d+1}$, where $d:=min\{k:p_k>0\}$ and $T_{d+1}$ is the homogeneous tree with offspring $d$.
Hence, $\FM(\Tr,\eta)$ is transient if we choose $\eta$ such that it holds
\[\bar{\eta} \leq \frac{1}{\rho(\Tr)} -1\, .\]
\end{proof}

Throughout this section, we shall make frequent use of several known results which we have assembled in Section \ref{appendix_frogs} below in the form of an appendix.

\subsection{No bushes, but  stretches} \label{subsec:nobushesbutstretches}
In the case $p_0+p_1>0$ a direct coupling as in the proof of Lemma  \ref{couplingstandard} does not allow us to prove transience since every non-trivial BMC is recurrent.
This is due to the existence of bushes or stretches in the Galton--Watson tree and the fact that the spectral radius of such trees is a.s.~equal to $1$, see Lemma  \ref{lemmaspectralradii}. We will start with dealing with stretches and then continue with treating bushes and stretches at the same time. 
The case of stretches uses a different method than in Lemma \ref{couplingstandard}. We modify the model, such that we wake up all frogs in a stretch, if the beginning of a stretch is visited for the first time. The awoken frogs are placed  according to the first exit measures (of a SRW) at the ends of this stretch. Moreover we send every frog entering a stretch  immediately to one of the ends of the stretch; again according to the exit measures. This makes it possible to consider the stretch as one vertex. However,  the original length of the stretch is important for the path measure  and the number of frogs. 

Let us explain why we did not succeed to construct a direct coupling between the frog model and a BMC.  The problem results from the frogs entering a stretch. These frogs leave the stretch according to the exit measure. The longer the stretch is, the higher the probability that a frog will return to the point from which it entered the stretch. Now, since the length of the stretches is not bounded,   this probability is not bounded away from $1$, and we can not dominate the frogs with an ``irreducible'' BMC. For this reason, we are comparing only the expectation, and not the whole distribution, of the returning frogs with the expectation of the returning particles of a suitable new BMC.
This new BMC will live on a truncated version of the Galton--Watson tree $\Tr$.
 We will truncate every stretch to a stretch of length at most  $N$. The resulting tree is denoted by $\Tr_{N}$, and the new BMC will live on the truncated tree $\Tr_{N}$ and will be denoted by $\BMC_{N}$. Most of our effort is then to choose the value of $N$ such that the following conditions hold. First, $N$ has to be sufficiently small such the $\BMC_{N}$ is transient on the truncated tree, and, second,  $N$ has to be sufficiently large such that  $\BMC_{N}$ ``dominates''  the frog model that lives on the larger tree $\Tr$.

\begin{proposition} \label{propositionstretches}
Consider a Galton--Watson measure $\GW$ with $0<p_1<1$, $p_0=0$ and mean $m>1$. We assume that  $d_{max}=\max\{i : p_{i}>0\}<\infty$ and set $d_{min}:=\min\{i\geq 2: p_{i}>0\}$. Then, for any choice of  $p_{1}$ there exist constants $c_{d}=c_{d}(p_{1})$ and $c_{\eta}=(c_{d}, d_{max})$ such that for $d_{min}>c_{d}$ the frog model $\FM(\Tr,\eta,\SRW)$ is transient $\GW$-almost surely (conditioned on $\Tr$ to be infinite) if $\bar{\eta} < c_\eta$.
\end{proposition}
\begin{proof}
Let $\Tr$ be an infinite realization of $\GW$. As $0<p_1<1$ we can consider $\Tr$ constructed according to $\ST\times \PER\times \GWbb$, see Section \ref{sectionGW}. Using this construction we label its vertices in the following way, see also Figure \ref{figurestretch}: 
\begin{itemize}
\item\textbf{label $\lbs$}: a vertex of degree $2$ with a mother vertex  of degree strictly larger than $2$;
\item\textbf{label $\les$}: a vertex of degree $2$ with a  child of  degree strictly larger than $2$;
\item \textbf{label $\ls$}: a vertex of degree $2$ with all two neighbours of degree $2$;
\item\textbf{label $\lno$}: a vertex of degree higher than $2$. 
\end{itemize} 
These labels help us to identify the stretches and their starting and end points. More precisely, a stretch is a path $[v_1,\ldots,v_n]$ where $v_1$ has label $\lbs$ and $v_n$ has label $\les$ and all vertices $v_i,i\in \{2,\ldots,n-1\}$, are labeled with $\ls$.
As mentioned above a $\BMC$ on a GW-tree with $0<p_1<1$ would a.s.\ be  recurrent. To find a dominating $\BMC$, which has a transient phase, we consider two  modified state spaces $\Tr'$ and $\Tr'_{N}$.

\subsubsection*{Construction  of  a dominating frog model $\FM'$ on $\Tr$ and $\Tr'$}
We modify the frog model in the following way.  Frogs in the new $\FM'$ behave as in $\FM$ on vertices that are not in stretches. Once a frog enters a stretch we add more particles in the following way. 
 Let $[v_1,\ldots,v_{\ell}]$ be a stretch of length $\ell=\ell_{v_1}$ and $u$ the mother vertex of $v_1$ and $w$ the child of $v_{\ell}$, see Figure \ref{figurestretch}. Here,  $v_{1}$ is the first vertex in a stretch, i.e.~a vertex with label $\lbs$.  {Now, if a first frog jumps on $v_1$, all frogs from the stretch are activated and placed on $u$ and $v$, respectively, according to their exit measures. For any later visit any frog entering the stretch is immediately  placed on $u$ or $v$ according to the exit measure of the stretch. The exit measures are solutions of a ruin problem.}
 Similar to the proof of Lemma \ref{couplingstandard}, we can couple $\FM$ and $\FM'$ such that
\[\nu \preccurlyeq \nu',\]
where $\nu'$ is the number of visits to the root in $\FM'$, and conclude that transience of $\FM'$ implies transience of $\FM$.

 \begin{figure}[h]
 \centering
 \begin{minipage}[t]{0.35\textwidth}
 \vspace*{0mm}
\includegraphics[width=1\textwidth]{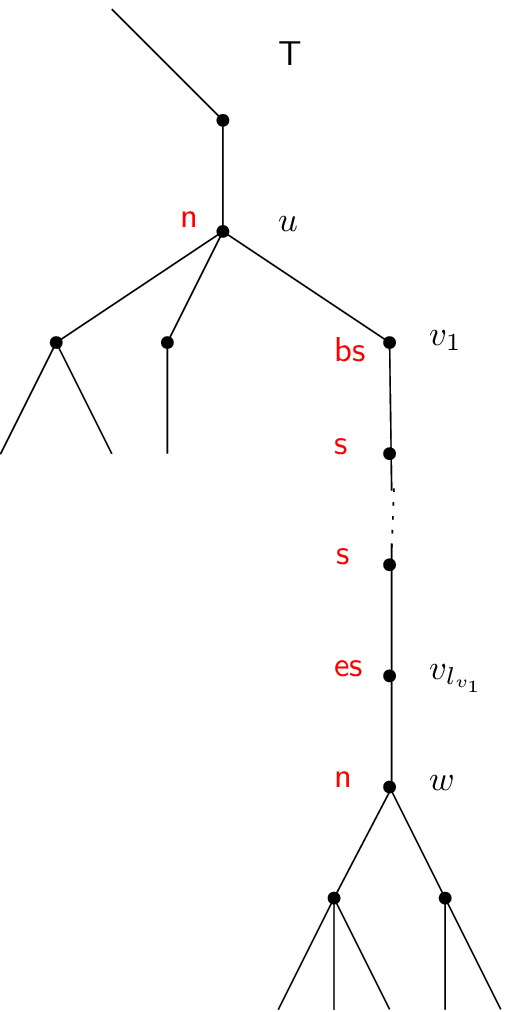}
\end{minipage}
\hfill
\begin{minipage}[t]{0.35\textwidth}
\vspace*{0mm}
\includegraphics[width=1\textwidth]{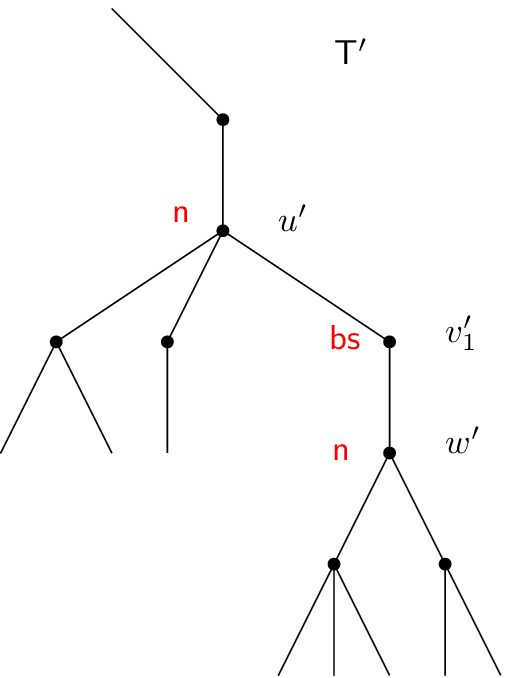}
\end{minipage}
\caption{A stretch $[v_1,\ldots,v_{\ell}]$ and its transformation to one vertex $v_1'$ in $\Tr'$.}
 \label{figurestretch}
\end{figure}

Concerning the stretches, in the definition of $\FM'$ only their ``exit measures'' play a role. The model $\FM'$ can therefore live on the tree $\Tr'$  constructed as follows. Let $[v_1,\ldots,v_{\ell}] \subset \Tr$ be a stretch and $w\in \Tr$ the child of $v_l$. Then, we merge the stretch into the vertex $v_1$ (with label $\lbs$). Hence, there is a single vertex of degree $2$ left in between vertices with higher degree, see Figure \ref{figurestretch}. We identify each vertex $v'\in\Tr'$ with its corresponding vertex $v \in \Tr$ due to this construction.  We can distinguish the vertices of $\Tr'$ into $V_{\mathsf{n}}:=\{v'\in \Tr'\mid v' \, \text{with label} \, \lno\}$ and $V_{\mathsf{s}}:=\{v'\in \Tr'\mid v' \, \text{with label} \, \lbs\}$. This modified state space $\Tr'$ corresponds to the first two stages, namely $\PER\times\GWbb$, in the construction of $\ST\times\PER\times\GWbb$. In other words, it has the same law as $\T_{\pa\times\tb\tg}$.
Moreover, the third step, i.e.~$\ST$, in the construction of the measure is encoded in the length of each stretch. 

We introduce the following  quantities. Let $\nu'(w')$ be the number of visits to $w'$ and $\nu'_n(w')$ the number of particles in  $w'$ at time $n$. Then, for a fixed realization $\Tr'$ let  
$\E^{\Tr'}_{v'}[\nu(w')]$ be the expected number of visits to $w'\in \Tr'$, when the  frog started in $v'\in \Tr'$. We also denote this as
\[m^{\Tr'}_{{\tiny{\FM}}'}(v',w'):=\E^{\Tr'}_{v'}[\nu(w')].\]
The expected number $m^{\Tr'}_{{\FM}'}(v',w')$ depends on the state space $\Tr'$ and we can look at the expected value 
\[m_{\tiny{\FM'}}^{\ST}(v',w'):=\E_{\tiny{\ST}}[m^{\Tr'}_{{\tiny{\FM}}'}(v',w')] \]
with respect to $\ST$ for $v',w' \in \Tr'$. Note here, that the measure $\ST$ has no impact on the underlying tree but only on the number of  frogs and the exit measure from the stretches.  Moreover, it holds that 
\begin{align}
m_{\tiny{\FM'}}^{\ST}(o',o') < \infty  \label{expectedexpectationfiniteFM}
\end{align}
implies
\begin{align}
m^{\Tr'}_{{\tiny{\FM}}'}(o',o') < \infty  \label{expectationfiniteFM}.
\end{align}

\subsubsection*{Construction  of dominating  $\BMC'$ on $\Tr'$}
In the next step we are going to define a branching Markov chain $\BMC'$ on $\Tr'$ such that \begin{align}
\E[\nu_{\tiny{\BMC'}}]<\infty \,\ \Rightarrow\,\ \E_{\tiny{\FM'}}[\nu']< \infty\, , \label{expectedrootreturncomparison} 
\end{align}
where $\nu_{\tiny{\BMC'}}$ is the number of returns to the root of the $\BMC'$.

 We recall that the length  $L$ of a stretch in the original tree $\Tr$ is geometrically distributed; $L\sim geo(p_1)$. Let $L_v, v\in \Tr$, denote this random stretch attached to a vertex $v$ with label $\lbs$.  The presence of arbitrarily long stretches prevents the existence of transient BMC on $\Tr$, see Lemma \ref{lem:unbranched}. Our strategy is to approximate the unbounded stretches with stretches of bounded size.  To do this, we define the tree $\Tr_{N}$ as a copy of $\Tr$ where each stretch of length larger than $N$ is replaced by a stretch of length $N$. Now,  for a given $p_{1}$,  we will find  some value $N\in \N$ such that the frog model on $\Tr$ can be ``bounded'' by a transient BMC on $\Tr_{N}$.
 
 \subsubsection*{Construction  of dominating  $\BMC_{N}$ on $\Tr_{N}$}
We define a BMC, called $\BMC_{N}$, on $\Tr_{N}$, with driving measure $\SRW$ and  offspring distribution  $\mu_k(v)=\Prob [\eta(v) +1=k]$ for each $v\in \Tr_N$. The $\BMC_{N}$, defined on $\Tr_{N}$, defines naturally a branching Markov chain $BMC'_{N}$ on $\Tr'$, where once a particle enters a stretch,  it produces offspring particles according to the exit-measures. This quantity is described by the first visit generating function
\begin{equation}
F_{N}(x,y |z):= \sum_{n=0}^{\infty} f^{(n)}_{N}(x,y) z^{n}. 
\end{equation}
The expected number of particles exiting a stretch of length $\ell$ in the entry vertex is given by $F_{\ell+1}(1,0\lvert \mu)$ while the expected number of particles exiting the stretch in the other vertex is given by $F_{\ell+1}(1,\ell+1\lvert \mu)$; 
 we refer to Subsection  \ref{sec:absBMC} for more details.
 The aim is now to find some integer $N$ such that $\BMC'_{N}$ is still transient and dominates (in $\ST$-expectation) the frog model $\FM'$. 
 
To find such a domination, we compare the mean number of visits  ``path-wise'' in $\FM'$ and $\BMC'_{N}$.  More precisely,  we want to express the quantity $\nu'_n(o')$ in terms of frogs following a specific path.  Let $\pa'$ be a path starting and ending at $o'$. A path of length $n\in \N$ looks like $\pa'=[o',p_1',p_2',\ldots, p'_{n-1},o']$ with $p'_i\in \Tr'$ and $p'_i\sim p'_{i+1}$ for each $i$. Let $\theta_k$ denote the $k$-th cut of a path, that is $\theta_k(\pa'):=[p_k',\ldots, o']$. We call a frog sleeping at some $p_i', 1\leq i \leq n-1$, {activated by frogs following the path $\pa'$} ($\mathsf{affb}_{\pa'}$), if inductively the frog was activated from a frog in $p_{i-1}$ that was activated by frogs following the path $\pa'$ or started at $p_1$ and followed $\pa'$. We denote by $\mathsf{affb}_{\pa'}(v',i)$ for the event that the $i$th frog  in $v'$ is $\mathsf{affb}_{\pa'}$. 
Additionally, for $i,j \in \N$ let  $S_j(v',i)$ denote the position of the $i$-th frog initially placed at $v' \in \Tr'$  after $j$ time steps after waking up. (Here we assume an arbitrary enumeration of the frogs at each vertex.)
Using this, $\nu'_n(o')$ is equal to
\begin{align*}
\left| \bigcup_{\lvert \pa'\lvert =n} \bigcup_{p'_i \in \pa'} \bigcup_{r\in \N}A(p'_{i},r,\pa)\right|\, .
\end{align*}
where
\begin{equation*}
A(p'_{i},r,\pa'):=\left\lbrace\exists  k :  \{S_j(p_i',r)\}_{j=0}^{n-k}=\theta_{k}(\pa') \, \text{and} \, \mathsf{affb}_{\pa'}(p'_{i},r)   \right\rbrace.
\end{equation*}
Now, we can rewrite
\begin{align}
&m_{\tiny{\FM'}}^{\ST}(o',o') =\E_{\tiny{\ST}}\left[\E^{\Tr'}_{o'}\left[\nu'\right]\right] = \E_{\tiny{\ST}}\Biggl[\E^{\Tr'}_{o'}\biggl[\ \sum\limits_{n=1}^{\infty} \nu'_n(o') \biggr] \Biggr] \\
&=\sum\limits_{n=1}^{\infty}\E_{\tiny{\ST}}\left[\E^{\Tr'}_{o'}\Biggl[\ \sum_{\lvert\pa'\lvert =n}  \biggl| \biggl\lbrace  \bigcup_{p'_i \in \pa} \bigcup_{r\in \N}A(p'_{i},r,\pa') \biggr\rbrace\biggr|\,  \Biggr]\right]  \\
&=\sum\limits_{n=1}^{\infty} \sum_{\lvert\pa'\lvert =n} \E_{\tiny{\ST}}\left[\E^{\Tr'}_{o'}\Biggl[ \, \biggl| \biggl\lbrace \ \bigcup_{p'_i \in \pa'} \bigcup_{r\in \N}A(p'_{i},r,\pa')\biggr\rbrace\biggr|\,  \Biggr]\right] \, \label{pathlengthn}
\end{align}  
by using the monotone convergence theorem.
For a given path $\pa'$ the term 
\[
\nu_{\tiny{\ST}}'(\pa'):=\E_{\tiny{\ST}}\left[\E^{\Tr'}_{o'}   \Biggl[ \, \biggr| \biggl\lbrace \ \bigcup_{p'_i \in \pa'} \bigcup_{r\in \N}A(p'_{i},r,\pa') \biggr\rbrace\biggr|\,  \Biggr]\right]
\] equals the expected number of frogs that were activated following the path and that follow the paths after their activation. In the same way as for the frog process we can define the expected number of particles $\nu_{\tiny{\BMC}}(\pa')$ for a BMC following a path $\pa'$. In the remaining part of the proof we construct a branching Markov chain $\BMC'_N$ such that 
\begin{equation}
\nu_{\tiny{\ST}}'(\pa')\leq \E\left[\nu_{\tiny{\BMC'_N}}(\pa')\right]
\end{equation} for all paths $\pa'$. Transience of the BMC then implies transiences of the frog model. 
The paths $\pa'$ are concatenations of three different types of vertex sequences. \textit{Type $1$} is a sequence that does not see any stretches. A sequence of \textit{type $2$} traverses a stretch, whereas a sequence of \textit{type $3$} visits a stretch but does not traverse it. We will split each path $\pa'$ into these three types and give upper bounds for (\ref{pathlengthn}) for each type separately. We have to take into account that multiple visits of the same sequence of vertices are not independent from each other. Here the frogs face in every visit the same length of a stretch. Hence, while taking the expectation over the length of the stretches, multiple visits of the same vertices have to be considered at the same time. Therefore, we give upper bounds of (\ref{pathlengthn}) for each combination of multiple visits. Then, we combine the results for a final upper bound of a mixed path.

For this purpose we consider for the $\BMC$ the mean number of particles created in stretches in $\Tr_{N}$. We consider the situation described in Figure \ref{figurestretch}. Let $\ell$ be the length of a stretch generated according to $\ST$. Such a stretch appears in $\Tr_N$ with probability $p_1^{\ell-1} (1-p_1)$ if $\ell < N-1$ and with probability $p^{N-1}$ if $\ell=N-1$. We denote by $m_{\tiny{\BMC'_N}}^{\Tr'}(p'_{i},p'_{i+1}) $ the expected number of particles arriving in $p'_{i+1}$ while starting in $p'_i$. Again we can look at the expectation with respect to $\ST$ 
\[m_{\tiny{\BMC'_N}}^{\ST}(p'_{i},p'_{i+1})=\E_{\tiny{\ST}}\left[ m_{\tiny{\BMC'_N}}^{\Tr'}(p'_{i},p'_{i+1})  \right],  \]
where $\ST$ impacts only the number of created particles and not the underlying tree.
We define the vertices $u$ and $w$ as absorbing and denote by $\eta_{N}(u)$ (resp.\  $\eta_{N}(w)$) the number of particles absorbed in $u$ (resp.\ in $w$), see also Section \ref{sec:A1}.

\textbf{Only visits of type 1}: We assume that $\pa'=[p'_{0},p'_1,p'_2,\ldots, p'_{n-1}, p'_{n}]$ only consists of sequences of type 1. Using the Markov property  we can bound
\begin{align}\label{pathlengthn2}
\nu'(\pa')& \leq \prod_{i=0}^{n-1 }m_{\tiny{\FM'}}^{\ST}(p'_{i},p'_{i+1}) = \prod_{i=0}^{n-1 }m_{\tiny{\BMC'_N}}^{\ST}(p'_{i},p'_{i+1}) ,
\end{align}
due to the choice of the $\BMC'_N$, see the paragraph after Equation (\ref{expectedrootreturncomparison}).

\textbf{Multiple visits of a stretch in sequences of type 2:} We assume that the path also has some sequences of type 2, see Figure \ref{figure:path_case3}.
 \begin{figure}[ht]
 \centering
\includegraphics[width=0.7\textwidth]{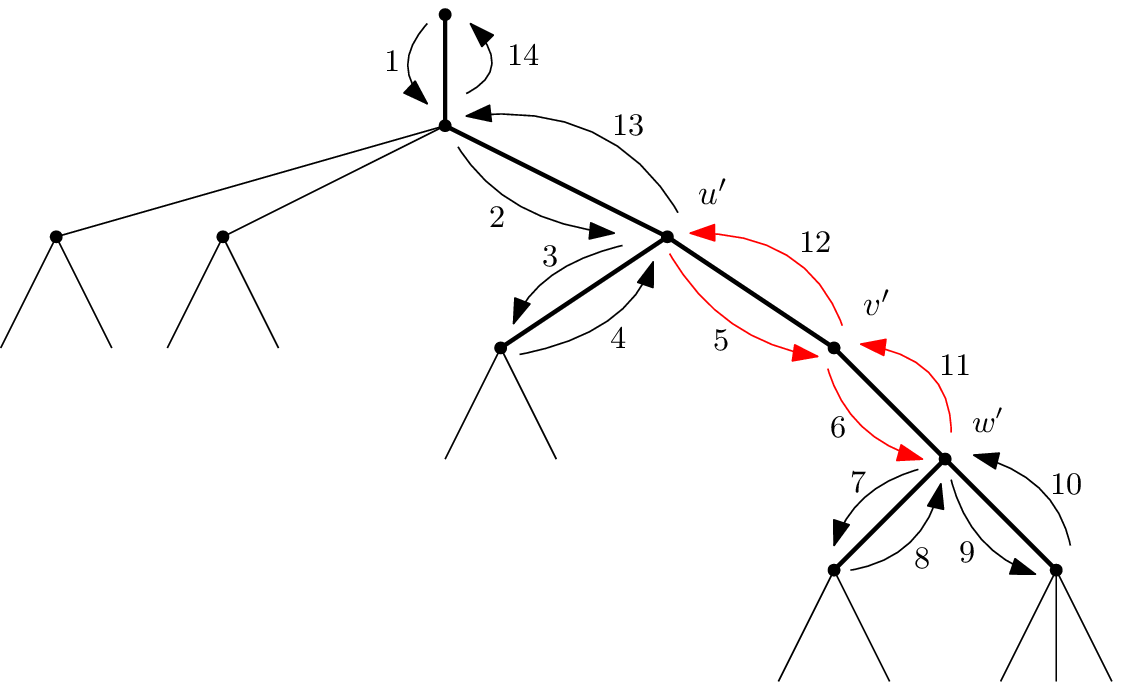}
\caption{A typical path with a sequence of type 2.}
 \label{figure:path_case3}
\end{figure}
An important observation is that every path from $o'$ to $o'$ that traverses a stretch in one direction has to traverse it in the other direction as well. Hence, such a path in $\Tr'$ of length $n$ has for example the form
\[
\pa'= [o', \underbrace{ p'_1,\ldots,p'_{i},u'}_{\text{degree} \geq 3}, v',\underbrace{w',p'_{i+4},\ldots,p'_{j},w'}_{\text{degree}\geq 3},v',\underbrace{u',p'_{j+4},\ldots,p'_{n-1}}_{\text{degree}\geq  3}, o'] ,
 \]
where $v'$ has degree $2$. We start with the case where the stretch is visited twice. The case of more visits will be an immediate consequence.

{In order to bound the expected number of frogs along a path we define
$m^{\Tr'}_{\tiny{\FM'}}(u'\rightarrow v' \rightarrow w') $ as the expected number of frogs that follow the path $[u', v', w']$ in $\FM'$ starting with one frog in $u'$. The modified frog model $\FM'$ is defined such that all frogs in the stretch are woken up and distributed at the end of the stretches if the starting vertex of the stretch is visited. In the case of traversing a stretch, this is dominated by the following modification: if the frog jumps on $v_{1}$ from $u$ the first time we start a BMC in $v_{1}$ with offspring distribution $\eta+1$ and absorbing states $u$ and $w$. The mean number of frogs absorbed in $u$ and $w$ can be calculated using Lemmata  \ref{lemma:Z-infty-formel} and \ref{lem:F}. This dominates $\FM'$ since we consider a path traversing the stretch. This means that all vertices in the stretch were visited in the new model, since some particles arrived in $w'$ and we can couple the sleeping frogs in $\FM'$ with the created particles in the stretch. We conclude by Lemma \ref{lemma:Z-infty-formel}  that}
\begin{equation}
m^{\Tr'}_{\tiny{\FM'}}(u'\rightarrow v' \rightarrow w') \leq \frac{1}{deg(u')} F_{\ell+1}(1, \ell+1 \mid {\bar{\eta}+1}),
\end{equation}
where $\ell=\ell_{v'}+1$ is the length of the total stretch (including the initial point of the stretch). To take into account that at the second traversal of the stretch no sleeping frogs are left in the stretch we define $m^{\Tr', 2+}_{\tiny{\FM'}}(w'\rightarrow v' \rightarrow u') $ as the expected number of frogs following $[w', v', u']$ with no frogs between $w'$ und $u'$. Hence,
\begin{align}
m^{\Tr'}_{\tiny{\FM'}}(u'\rightarrow v' \rightarrow w')  m^{\Tr', 2+}_{\tiny{\FM'}}(w'\rightarrow v' \rightarrow u') &\leq \frac{1}{deg(u')} F_{\ell+1}(1, \ell +1 \mid {\bar{\eta}+1}) \frac{1}{deg(w')} \frac{1}{\ell+1} \\
& \leq  \frac{1}{deg(u')} F_{\ell+1}(1, \ell+1 \mid  {\bar{\eta}+1})^{2} \frac{1}{deg(w')}. \label{eq:type2first}
\end{align}
Note that the last term equals the mean number of particles ending at $u'$
in  $\BMC'_{\ell}$, starting with one particle in $u'$ and following the path $[u', v', w',v',u']$.   

Now, we want to find an $N\in \N$ such that we can dominate a visit of type 2 to a stretch by the $\BMC_N$, that is, we  want $N \in \N$ such that it holds
\begin{align}
m^{\Tr'}_{\tiny{\FM'}}(u'\rightarrow v' \rightarrow w')  & m^{\Tr', 2+}_{\tiny{\FM'}}(w'\rightarrow v' \rightarrow u')\\
& \leq  \frac{1}{deg(u')} F_{N+1}(1, N+1 \mid  {\bar{\eta}+1})^{2} \frac{1}{deg(w')} \label{eq:type2N}
\end{align}
for all possible lengths $\ell$ of a stretch in $\Tr$. For this purpose we consider
\begin{align}
m^{\Tr'}_{\tiny{\FM'}}(u'\rightarrow v' \rightarrow w')  m^{\Tr', 2+}_{\tiny{\FM'}}(w'\rightarrow v' \rightarrow u') &\leq \frac{1}{deg(u')} \left(\frac{\ell \bar{\eta} }2 + \frac{1}{\ell+1} \right) \frac{1}{deg(w')} \frac{1}{\ell+1}.  
\end{align}
We have a lower bound for the right hand side of (\ref{eq:type2N}) from Lemma \ref{lem:Fapprox}. Omitting the transition probabilities from $u'$ to $v'$ and from $w'$ to $v'$ we obtain for $\varphi= \arccos(\frac{1}{1+\bar{\eta}}) >0$  that
\begin{align}
F_{N+1}(1, N+1 \mid  {\bar{\eta}+1})^{2} &\geq \left(\frac{1}{N+1} \left(1+ \frac{ \left(2N+N^2\right)\varphi^{2}}{3!}\right) \right)^2 \\
& \geq \left(\frac{1}{N+1}\right)^2 +  \frac13 N \varphi^{2}.
\end{align}
Hence, we have to find $\bar \eta$ and $N$ such that 
\begin{equation}\label{eq:type2N1}
 \frac{\ell \bar{\eta} }2  \frac{1}{\ell+1} + \left(\frac{1}{\ell+1} \right)^{2} \leq \left(\frac{1}{N+1}\right)^2 +  \frac13 N \varphi^{2}.
\end{equation}
for all $\ell >N$. Note that we want that the $\BMC_{N}$ is transient and therefore that 
\begin{equation}
\varphi < \frac{\arccos \left(\frac{{2}\sqrt{d}}{d+1}\right)}{N+1},
\end{equation}
with $d=\min\{i\geq 2: p_i> 0\}$; this is a consequence of Lemma \ref{lemmaspectralradii},  Theorem   \ref{thm:rhoN}, and Lemma \ref{lem:specapprox}.  Let $\gamma  < \arccos\left( \frac{{2}\sqrt{d}}{d+1}\right)$ be such that
\begin{equation}
\bar \eta = \left( \left( \cos \frac{\gamma}{N+1}   \right)^{-1}  -1  \right).
\end{equation}
Inequality (\ref{eq:type2N1}) holds if 
\begin{equation}\label{eq:type2N2}
\frac12  \left( \left( \cos \frac{\gamma}{N+1}   \right)^{-1}  -1  \right) \leq \frac13 N 
\left(\frac{\gamma}{N+1}\right)^{2}.
\end{equation}
Using L'Hospital's rule we see that for any choice of $\gamma$ the inequality above is true for $N$ sufficiently large. 	
 In the case of multiple type 2 visits the proof is a immediate consequence of only one type 2 visit. For $\ell > N$ it holds 
\begin{align}
m^{\Tr'}_{\tiny{\FM'}}(u'\rightarrow &v' \rightarrow w')  m^{\Tr', 2+}_{\tiny{\FM'}}(w'\rightarrow v' \rightarrow u')^k m^{\Tr',2+}_{\tiny{\FM'}}(u'\rightarrow v' \rightarrow w')^{k-1} \\ 
&\leq \left(\frac{1}{deg(u')}\right)^k\left(\frac{\ell \bar{\eta} }2 + \frac{1}{\ell+1} \right) \left(\frac{1}{deg(w')}\right)^k \left(\frac{1}{\ell+1}\right)^{2k-1}
\\ 
&\leq \left(\frac{1}{deg(u')}\right)^k\left(\frac{\ell \bar{\eta} }2 + \frac{1}{\ell+1} \right) \left(\frac{1}{deg(w')}\right)^k \left(\frac{1}{\ell+1}\right)  \left(\frac{1}{N+1}\right)^{2k-2} \label{eq:typetwomultiple}
\end{align}
The probability that a simple random walk enters the stretch of length $N+1$ and reaches the other side is $1/(N+1)$. This quantity is naturally dominated by $F_{N+1}(1, N+1 \mid  {\bar{\eta}+1})$  since in the BMC at least one particle moves according to the simple random walk in the stretch. Using this we give an upper bound for (\ref{eq:typetwomultiple}) that holds for $N$ sufficiently large:
\begin{align}
&\left(\frac{1}{deg(u')}\right)^k  \left(\frac{\ell \bar{\eta} }2 + \frac{1}{\ell+1} \right) \left(\frac{1}{deg(w')}\right)^k \left(\frac{1}{\ell+1}\right)  \left(\frac{1}{N+1}\right)^{2k-2} \\
& \leq \left(\frac{1}{deg(u')}\right)^k  \left(\frac{\ell \bar{\eta} }2 + \frac{1}{\ell+1} \right) \left(\frac{1}{deg(w')}\right)^k  \left(\frac{1}{\ell+1}\right)    F_{N+1}(1, N+1 \mid  {\bar{\eta}+1})^{2k-2}  \\
& \leq \left(\frac{1}{deg(u')}\right)^k  \left(\frac{1}{deg(w')}\right)^k F_{N+1}(1, N+1 \mid  {\bar{\eta}+1})^{2k}  \\
& \leq m^{\Tr'}_{\tiny{\BMC'_N}}(u'\rightarrow v' \rightarrow w')^k m^{\Tr'}_{\tiny{\BMC'_N}}(u'\rightarrow v' \rightarrow w')^k,
\end{align}\label{eq:multiple:type2}
where $m^{\Tr'}_{\tiny{\BMC'_N}}(u'\rightarrow v' \rightarrow w')$ is the  expected number of particles following the path $[u', v', w']$ in $\BMC'_N$ starting with one particle in $u'$. 
Using the aforegoing estimate, we can bound $\E^{\ST}[m^{\Tr'}_{\tiny{\FM'}}(u'\rightarrow v' \rightarrow w')  m^{\Tr', 2+}_{\tiny{\FM'}}(w'\rightarrow v' \rightarrow u')^k m^{\Tr',2+}_{\tiny{\FM'}}(u'\rightarrow v' \rightarrow w')^{k-1} ]$  by $\E_{\ST}[m^{\Tr'}_{\tiny{\BMC'_N}}(u'\rightarrow v' \rightarrow w')^k m^{\Tr'}_{\tiny{\BMC'_N}}(w'\rightarrow v' \rightarrow u')^k]$ for  $N\in \N$ sufficiently large.
Moreover  the stretches are independently generated.  We obtain by induction for different sequences of type 2 that:
\begin{align}\label{pathlengthnganzBMC}
\nu'(\pa')& \leq\prod_{i=0}^{n-1 }m_{\tiny{\BMC'_N}}^{\ST}(p'_{i},p'_{i+1}) 
\end{align}
using the short notation $\E_{\ST}[m^{\Tr'}_{\tiny{\BMC'_N}}(u', w')]$ instead of $\E_{\ST}[m^{\Tr'}_{\tiny{\BMC'_N}}(u'\rightarrow v' \rightarrow w')]$.

\textbf{Multiple visits of a stretch in sequences of type 2 and 3:} 
We handle this situation in three steps. In the first we assume, that a sequence of vertices is only visited once in the manner of type 3. Secondly, we  treat a sequence of a path which visits a stretch more than once in the manner of type 3. Lastly, we study sequences which are visited by type 2 and type 3 sequences. There, we have to distinguish between the type of the first visit of the sequence.

We start with the first part. 
We assume that the path $\pa$ of length $n$ contains a  sequence of type 3, that is $p'_{i_j}=v'$,$i_j \in \{1,\ldots,n\}$, of degree $2$ and $p_{i_j-1}=p_{i_j+1}=u'$, see Figure \ref{figure:path_case2}. 
 \begin{figure}[ht]
 \centering
\includegraphics[width=0.70\textwidth]{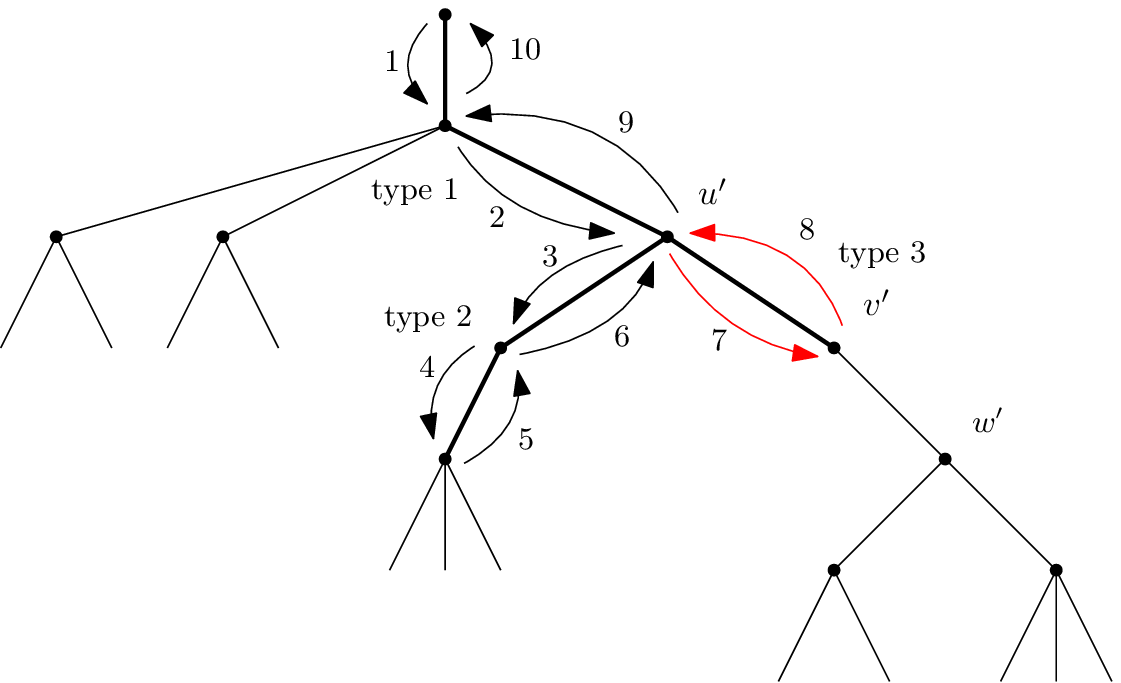}
\caption{A typical path with sequences of type 1,2 and 3.}
 \label{figure:path_case2}
\end{figure}
This means that the frogs in $\FM'$ did not pass the stretch completely. We call these parts of the path stretchbits.  A typical path $\pa$ in this case can be for example

\[
\pa'= [o',\underbrace{p'_1,\ldots,p'_{i_1-2}}_{\text{type 1,2}},\overbrace{u', v',u'}^{\text{type 3}},\underbrace{p'_{i_1+2},\ldots, p'_{n-1}}_{\text{type 1,2}}, o']\, .
\]

We define  $m^{\Tr'}_{\tiny{\FM'}}(u'\rightarrow v' \rightarrow u') $ as the expected number of frogs that follow the path $[u', v', u']$ in $\FM'$ starting with one frog in $u'$. Then 

\begin{align}\label{eq:caseupperbound}
m^{\Tr'}_{\tiny{\FM'}}(u'\rightarrow v' \rightarrow u')   & \leq \frac{1}{deg(u')}\left(\frac{\ell \bar{\eta} }2 + \frac{\ell}{\ell+1} \right).
\end{align}
Recall that the distribution of the total stretch length $L=\ell_{v_1}+1$ is geometric:
$$
\Prob(L=\ell)=p_{1}^{\ell-1} (1-p_{1}),~\forall ~\ell\geq 1.
$$ 
Hence, integrating (\ref{eq:caseupperbound}) with respect to $\ST$ yields 
\begin{align}
m^{\ST}_{\tiny{\FM'}}(u'\rightarrow v' \rightarrow u')   & \leq  \frac{1}{deg(u')} \sum_{\ell=1}^{\infty} \left(\frac{\ell \bar{\eta}}{2} + \frac{\ell}{\ell+1} \right) p_{1}^{\ell-1} (1-p_{1})   \\
& = \frac{1}{deg(u')} \left( \frac{\bar{\eta}}{2(1-p_{1})}+ \sum_{\ell=1}^{\infty} \left( \frac{\ell}{\ell+1} \right) p_{1}^{\ell-1} (1-p_{1}) \right).  
\end{align}
 Let $d=\min\{i\geq 2: p_i> 0\}$. A stretch of length $\ell$ is equivalent to an unbranched path of length $\ell +1$ in Section \ref{subsec:spectralradiusoftrees}. As we only allow a maximum stretch length $N$ in case of $\BMC_N$, we obtain at maximum an unbranched path of length $N+1$. Then, using Lemma \ref{lemmaspectralradii},  Theorem   \ref{thm:rhoN}, and Lemma \ref{lem:specapprox}  the spectral radius $\rho(P_{N+1})$ on the absorbing stretch piece  of length $N+1$ satisfies
\begin{equation} \rho(P_{N+1}) <   \cos\left( \frac{\arccos \left(\frac{{2}\sqrt{d}}{d+1}\right)}{N+1}\right).
\end{equation}
Furthermore,
\begin{align}\label{eq:caseBMC}
& m^{\Tr'}_{\tiny{\BMC'_N}}(u'\rightarrow v' \rightarrow u') =  \frac{1}{deg(u')} F_{\ell+1}(1, 0 | \bar \mu).
\end{align}
We now choose 
\begin{equation}
\bar \mu= \cos\left( \frac{\arccos \left(\frac{{2}\sqrt{d}}{d+1}\right)}{N+1}-\varepsilon\right)^{-1}  
\end{equation} for some sufficiently small $\varepsilon>0$ and define
$$
g(\ell) = F_{\ell}\left(1, 0 \middle| \bar{\mu}\right)<\infty.
$$
Observe here that, since $\bar \mu < \frac{1}{\rho(\Tr')}$, the $\BMC'_N$ with mean offspring $\bar \mu$ is not only transient but it also holds that $\E_{\BMC'_N}[\nu]<\infty$, see Chapter 5.C in \cite{woess}.
Now, integrating equation (\ref{eq:caseBMC}) with respect to $\ST$ yields
\begin{align}
& m^{\ST}_{\tiny{\BMC'_N}}(u'\rightarrow v' \rightarrow u') =   \frac{1}{deg(u')} \sum_{\ell=1}^{N-1} g(\ell+1) p_{1}^{\ell-1} (1-p_{1}) + g(N+1) p_{1}^{N-1}.
\end{align}
We now look for $\bar{\eta}$ sufficiently small and $N$ sufficiently large such that 
\begin{align}
&\left( \frac{\bar{\eta}}{2}\left(\frac{1}{1-p_1}\right)+ \sum_{\ell=1}^{\infty} \left( \frac{\ell}{\ell+1} \right) p_{1}^{\ell-1} (1-p_{1}) \right)    \\
& \quad \quad  < \sum_{\ell=1}^{N-1} g(\ell+1) p_{1}^{\ell-1} (1-p_{1}) + g(N+1) p_{1}^{N-1}. \label{eq:comparisontype3}
\end{align}
In order to achieve this last inequality, it suffices to find an $N$ such that
\begin{align}
\sum_{\ell=1}^{\infty} \left( \frac{\ell}{\ell+1} \right) p_{1}^{\ell-1} (1-p_{1}) < \sum_{\ell=1}^{N-1} g(\ell+1) p_{1}^{\ell-1} (1-p_{1}) + g(N+1) p_{1}^{N-1}. \label{eq:stretchesfirstcomp}
\end{align}
By Lemma \ref{lem:Fapprox} we can bound the right hand side from below by 
\begin{align}
\sum_{\ell=1}^{N-1} \frac{\ell}{\ell+1} \left(1+ \frac{ \left(1+2\ell\right)\varphi^{2}}{3!}\right) p_{1}^{\ell-1} (1-p_{1}) + g(N+1) p_{1}^{N} \,  
\end{align}
where $\varphi=\arccos(1/ \bar\mu)$.
This reduces (\ref{eq:stretchesfirstcomp}) to:
\begin{align}
\sum_{\ell=N}^{\infty} \left( \frac{\ell}{\ell+1} \right) p_{1}^{\ell-1} (1-p_{1}) < \sum_{\ell=1}^{N-1} \frac{\ell}{\ell+1} \left(\frac{ \left(1+2\ell\right)\varphi^{2}}{3!}\right) p_{1}^{\ell-1} (1-p_{1}) + g(N+1) p_{1}^{N-1}\, . \label{eq:approx}
\end{align}
The left hand side of (\ref{eq:approx}) decays exponentially in $N$ while the first part of the right hand side has polynomial decay in $N$ having the choice of $\varphi$ in mind. Therefore, there exists some $N$ such that (\ref{eq:approx}) is verified.

We continue with the second part, where a sequence of the path faces multiple type 3 visits.
If a frog makes a second type 3 visit to an already woken up stretch, this frog encounters no new frogs and returns to $u'$ almost surely. This follows for every other visit of type 3. Hence, conditioning the frog upon not making another type 3 visit to a stretch has no influence on the possible frogs returning to the root and consequently on transience and recurrence. We will call this model $\FM''$. But we notice that the path measure changes when we change to $\FM''$:
\begin{align}
\Pa[u'\rightarrow y'\, \lvert \, \text{no visit to} \, v'  ]=\frac{1}{deg(u')-1}  
\end{align}
where $y'$ is any neighbour of $u'$ apart from $v'$. Since the path measure of $\BMC_N$ is unchanged we have to compare 
\[m_{\tiny{\BMC'_N}}^{\Tr'}(u',y')=\frac{\bar{\mu}}{deg(u')}\] 
and
  \[m_{\tiny{\FM''}}^{\Tr'}(u',y')=\frac{1}{deg(u')-1}\, \] 
as $u'$ was visited already by assumption and obtain
\begin{align}
\frac{1}{deg(u')-1} \leq  \frac{\bar{\mu}}{deg(u')} \Longleftrightarrow \frac{deg(u')}{deg(u')-1} \leq {\bar{\mu}} \, .  
\end{align}
We conclude for the mean offspring $\bar{\mu}$ of $\BMC_N$ that a necessary condition for our majorization is 
\begin{align}
\frac{d_{min}+1}{d_{min}}\leq {\bar{\mu}} \, \label{eq:erasecomp}
\end{align}
with $d_{min}:=\min\{k\geq 2:p_k>0\}$ is a necessary condition for our majorization. Using the new model $\FM''$ we are left with only the first visit of type 3 to the stretch. As we have seen before, there is a $N$ such that (\ref{eq:approx}) holds.

Now, we will treat the third part, where we allow multiple visits of type 2 and 3 to a sequence of vertices. 
We want to erase again multiple visits of type 3 of a stretch and assume, that (\ref{eq:erasecomp}) holds, such that the $\BMC_N$ dominates the conditioned path. Then it remains to deal with either a first visit of type 2 or a first visit of type 3 and multiple visits of type 2. If the first visit is of type 2, we can bound the frog model by using (\ref{eq:multiple:type2}) additionally to (\ref{eq:erasecomp}). 

If the first visit is of type 3, and we have apart from other visits of type 3 (which will be erased and bounded using (\ref{eq:erasecomp})) $k$ visits and returns of type 2, we obtain
\begin{align}
&\E_{\ST}\left[m^{\Tr'}_{\tiny{\FM'}}(u'\rightarrow v' \rightarrow u')  m^{\Tr', 2+}_{\tiny{\FM'}}(w'\rightarrow v' \rightarrow u')^k m^{\Tr',2+}_{\tiny{\FM'}}(u'\rightarrow v' \rightarrow w')^{k}\right]    \\
&=  \left(\frac{1}{deg(u')}\right)^{k+1} \E_{\ST}\left[\left(\frac{\ell \bar{\eta}}{2} +\left(\frac{\ell}{\ell+1}\right) \right)\left(\frac{1}{\ell +1}\right)^{2k}  \right] \left(\frac{1}{deg(w')}\right)^k \, .  
\end{align}
For the upcoming equations we omit the factors of the transitions probabilities from $u'$ to $v'$ and from $w'$ to $v'$. These probabilities are the same for the BMC  and do not play a role for the comparison with the frog model. Then we get:
\begin{align}
&\sum_{\ell=1}^{\infty} \frac{\ell \bar{\eta}}{2} \left(\frac{1}{\ell+1}\right)^{2k}p_1^{\ell-1} (1-p_1)+ \sum_{\ell=1}^{\infty} \left(\frac{\ell}{\ell+1}\right) \left(\frac{1}{\ell +1}\right)^{2k}  p_1^{\ell-1} (1-p_1)    \\
&\leq\frac{ \bar{\eta} }{2} \left( \frac{1-p_1}{p_1^2}\right) \sum_{\ell=1}^{\infty} \frac{p_1^{\ell+1}}{(\ell+1)^{2k-1}} + \sum_{\ell=1}^{\infty} \left(\frac{\ell}{\ell+1}\right) \left(\frac{1}{\ell +1}\right)^{2k}  p_1^{\ell} (1-p_1).\label{eq:finalmultipletype23FM}
\end{align}
For the BMC we have the following identities as before:
\begin{align}
& \E_{\ST} \left[m^{\Tr'}_{\tiny{\BMC'_N}}(u'\rightarrow v' \rightarrow u') m^{\Tr'}_{\tiny{\BMC'_N}}(u'\rightarrow v' \rightarrow w')^k m^{\Tr'}_{\tiny{\BMC'_N}}(w'\rightarrow v' \rightarrow u')^k \right]     \\
&=   \left(\frac{1}{deg(u')}\right)^{k+1} \E_{\ST}\left[F_{\ell+1}(1,0\lvert \bar{\mu})^{2} F_{\ell+1}(\ell,\ell+1\lvert \bar{\mu})^{2k}\right] \left(\frac{1}{deg(w')}\right)^{k}  
\end{align}
By Lemma \ref{lem:Fapprox} (and again omitting the transitions probabilities) this is greater or equal to 
\begin{align}
&\sum_{\ell=1}^{N-1} \frac{\ell}{\ell+1} \left(\frac{1}{\ell+1}\right)^{2k} \left(1+ \frac{ \left(1+2\ell\right)\varphi^{2}}{3!}\right) p_{1}^{\ell-1} (1-p_{1})    \\
 &\qquad  \qquad + \frac{N}{N+1} \left(\frac{1}{N+1}\right)^{2k} \left(1+ \frac{ \left(1+2N\right)\varphi^{2}}{3!}\right) p_{1}^{N}.
 \label{eq:finalmultipletype23BMC}
\end{align}
We want to show that we can choose for each $p_1$ and $N$ an $\eta$ such that the following holds for all $k\geq 1$: 
\begin{align}
&\frac{ \bar{\eta} }{2} \left( \frac{1-p_1}{p_1^2}\right) \left(  \sum_{\ell=1}^{N-1} \frac{p_1^{\ell+1}}{(\ell+1)^{2k-1}}+ \sum_{\ell=N}^{\infty} \frac{p_1^{\ell+1}}{(\ell+1)^{2k-1}}\right)\label{eq:1}\\
&\qquad + \sum_{\ell=1}^{N-1} \left(\frac{\ell}{\ell+1}\right) \left(\frac{1}{\ell +1}\right)^{2k}  p_1^{\ell-1} (1-p_1)\label{eq:2} \\
&\qquad + \sum_{\ell=N}^{\infty} \left(\frac{\ell}{\ell+1}\right) \left(\frac{1}{\ell +1}\right)^{2k}  p_1^{\ell-1} (1-p_1) \label{eq:3}\\
&\leq \sum_{\ell=1}^{N-1} \frac{\ell}{\ell+1} \left(\frac{1}{\ell+1}\right)^{2k} p_{1}^{\ell-1} (1-p_{1}) \label{eq:4}\\
&\qquad + \sum_{\ell=1}^{N-1} \frac{\ell}{\ell+1} \left(\frac{1}{\ell+1}\right)^{2k} \left(\frac{ \left(1+2\ell\right)\varphi^{2}}{3!}\right) p_{1}^{\ell-1} (1-p_{1})   \label{eq:5}\\
 &\qquad + \frac{N}{N+1} \left(\frac{1}{N+1}\right)^{2k} p_{1}^{N-1}+ \frac{N}{N+1} \left(\frac{1}{N+1}\right)^{2k} \left(\frac{ \left(1+2N\right)\varphi^{2}}{3!}\right) p_{1}^{N-1}. \label{eq:6}
\end{align}
The second part of the left hand side, (\ref{eq:2}), is equal to the first part, (\ref{eq:4}), on the right hand side. Next we compare the third part of the left,  (\ref{eq:3}), to the third part on the right, (\ref{eq:6}). We notice that the function $\left(\frac{\ell}{\ell+1}\right) \left(\frac{1}{\ell +1}\right)^{2k}$ is monotonically decreasing in $\ell$ and thus
\begin{align}
& \sum_{\ell=N}^{\infty} \left(\frac{\ell}{\ell+1}\right) \left(\frac{1}{\ell +1}\right)^{2k}  p_1^{\ell-1} (1-p_1)\leq \left(\frac{N}{(N+1)^{2k+1}}\right) \sum_{\ell=N}^{\infty}  p_1^{\ell-1} (1-p_1)   \\
&= \left(\frac{N}{(N+1)^{2k+1}}\right) p_1^{N-1} \, .  
\end{align}
Now, we consider the remaining term on the left hand side, (\ref{eq:1}), and the second of the right hand side, (\ref{eq:5}). We start with giving an upper bound for the second sum in (\ref{eq:1}):
\begin{align}
\sum_{\ell=N}^{\infty} \frac{p_1^{\ell+1}}{(\ell+1)^{2k-1}} &\leq\left(\frac{1}{N+1}\right) ^{2k-1}\sum_{\ell=N-1}^{\infty} \frac{p_1^{\ell}}{(1-p_1)}  = \left(\frac{1}{N+1}\right) ^{2k-1}\frac{p_1^{N-1}}{(1-p_1)}.  
\end{align}
The second term of the right hand side, (\ref{eq:5}),  can be transformed into
\begin{align}
&\sum_{\ell=1}^{N-1} \frac{\ell}{\ell+1} \left(\frac{1}{\ell+1}\right)^{2k} \left( \frac{ \left(1+2\ell\right)\varphi^{2}}{3!}\right) p_{1}^{\ell-1} (1-p_{1})   \\
&\quad \geq \frac{1-p_{1}}{p_{1}^2} \left(\frac{\varphi^2}{3! (N+1)^2}\right) \sum_{l=1}^{N-1} \left(\frac{2\ell^2}{(\ell+1)^2}\right) \left(\frac{p_1^{\ell+1}}{(\ell+1)^{2k-1}}\right)   \\ 
&\quad \geq \frac{1-p_{1}}{p_{1}^2} \left(\frac{ (\arccos(1/ \bar\mu))^2}{3! (N+1)^2}\right) \frac{1}{2} \sum_{\ell=1}^{N-1} \left(\frac{p_1^{\ell+1}}{(\ell+1)^{2k-1}}\right).  
\end{align}
We have that $(\ref{eq:1})<(\ref{eq:5})$ if 
\begin{align}
&\frac{ \bar{\eta} }{2} \left( \frac{1-p_1}{p_1^2}\right) \left(  \sum_{\ell=1}^{N-1} \frac{p_1^{\ell+1}}{(\ell+1)^{2k-1}}+\left(\frac{1}{N+1}\right)^{2k-1} \frac{p_1^{N-1}}{(1-p_1)}\right)   \\
&\quad \leq \left(\frac{1-p_{1}}{p_{1}^2} \right)\left(\frac{(\arccos(1/ \bar\mu))^2}{3! (N+1)^2}\right) \frac{1}{2} \sum_{\ell=1}^{N-1} \left(\frac{p_1^{\ell+1}}{(\ell+1)^{2k-1}}\right).  
\end{align} For all choices of $p_{1}$ and $N\in \N$ we can now find $\bar{\eta}$ sufficiently small such that the latter inequality is verified for all $k \in \N$.

\subsubsection*{Summary} 
We summarize all the conditions on $\eta$ and $\bar{\mu}$ such that we can find a dominating transient $\BMC$ for a given frog model $\FM$  in the case when stretches come up:

\begin{enumerate} 
\item $1+\bar{\eta} < \bar{\mu}$;
\item $\frac{d_{min}+1}{d_{min}}\leq \bar{\mu} $;
\item Choosing $\eta$ such that $\bar{\eta}$ is small enough such that there exists some $N\in\N$ such that  (\ref{eq:type2N2}) and (\ref{eq:comparisontype3}) hold;
\item Choosing $\eta$ such that for given $p_1$ and the previously selected $N$ the inequality  (\ref{eq:1})- (\ref{eq:6}) holds;
\item $\bar{\mu} <   \left(\cos \frac{\arccos \left(\frac{{2}\sqrt{d_{min}}}{d_{min}+1}\right)}{N+1}\right)^{-1}$.
\end{enumerate}
In other words, for every  $p_{1}>0$ there exists some $N$ big enough such that if 
\begin{equation} \label{eq:final_stretches}
\frac{d_{min}+1}{d_{min}} < \left(\cos \frac{\arccos \left(\frac{{2}\sqrt{d_{min}}}{d_{min}+1}\right)}{N+1}\right)^{-1}
\end{equation}
there exists some small $\bar{\eta}>0$ and some $\BMC'_N$ with mean offspring larger than $1$ such that $\E[\nu_{\tiny{\BMC'_N}}]<\infty$ and  
\begin{equation}
\nu_{\tiny{\ST}}'(\pa')\leq \E[\nu_{\tiny{\BMC'_N}}(\pa')]
\end{equation} for all paths $\pa'$ and $\GW$-a.a.~trees $\Tr'$.  Finally, we found that $\nu'<\infty$ $\fmm$-a.s.~for $\GW$-a.a.~trees and hence $\nu<\infty$ $\fmm$-a.s.~for $\GW$-a.a.~trees. The existence of the constant $c_{\eta}$ follows from the $0$--$1$-law of transience. 

The existence of a transient phase is guaranteed since for all $N$ there exists $d_{min}$ fulfilling (\ref{eq:final_stretches}) as (\ref{eq:final_stretches}) is equal to 
\[N+1 < \frac{\arccos \left(\frac{{2}\sqrt{d_{min}}}{d_{min}+1}\right)}{\arccos\left(\frac{d_{min}}{d_{min}+1}\right)}\, \]
and the right hand side converges to $\infty$ for $d_{min} \rightarrow \infty$. 
\end{proof}


\subsection{Bushes and {possible} stretches} \label{subsec:bushesandstretches}
It is left to prove the main theorem of this paper where we allow $p_0>0$. The proof starts with the following modification: Once a frog visits a vertex $v \in \Tr$ with bushes attached, all frogs in the bushes are woken up and placed at $v$. This is equivalent to changing the number of frogs in $v$ and conditioning the frogs not to enter the bush. The erasure of the bushes does not change the transience behaviour of the process. Following this procedure, we end up with trees with stretches and without bushes and we can then apply the proof of  Proposition \ref{propositionstretches}.


\begin{proof}[Theorem \ref{main}]
We assume that  $p_0>0$ and start with explaining how we  remove the bushes.

\subsubsection*{Removing bushes from $\Tr$}

Every infinite  GW-tree can be seen  as a multitype GW-tree $\bar{\T}$ with types $\tg$ and $\tb$, see Section \ref{sectionGW}. We denote by $\Tr$ a realization of $\GW$ conditioned to be infinite. Moreover we recall that our GW-tree has bounded offspring: there is a $K=d_{max}<\infty$ such that $Y_i^{(n)}\leq d_{max}$ for all $i,n \in \N$. Therefore, every vertex which is part of a geodesic stretch can have at most $K-1$  finite bushes attached. 

To start with, we modify the original frog model $\FM$. If a frog visits a vertex $v \in \Tr$ with attached bushes for the first time, then immediately all frogs from the bushes attached to $v$ wake up and are placed at $v$. As $K+1$ is the maximum degree of the tree, we know that there are at most $K-1$ bushes attached to a vertex of type $\tg$. More formally, let $v_i, i\in \{1,\ldots,k\}$ and $k\leq K-1$, the vertices of type $\tb$ adjacent to $v$ and let $G_{v_i}$ denote the random bush starting with root $v_i \in \Tr$. Then, there will be $\eta^{*}(v):=\sum_{i=1}^{k}\sum_{w\in G_{v_i}} \eta(w)$  frogs in vertex $v$ with attached bushes and $\eta^{*}(u):=\eta_u$ frogs in a vertex of type $\tg$ with no attached bushes.  The bushes  $G_{v_i}$ are i.i.d.\ distributed like a subcritical GW-process with generating function $\tilde{f}$, see Section \ref{sectionGW}, and the expected size of $G_{v_i}$ is finite.  Conditioning the frog model on not entering bushes we obtain different transition probabilities for each frog. Let $v$ be a vertex with neighboured bushes, $v_1,\ldots,v_k$, $k\leq K-1$ the attached roots of bushes and $w_1,\ldots,w_d$, $d\leq K+1-k$ its neighbours of type $\tg$.
 Then we obtain 
\begin{align}
\Pa[v \rightarrow w_i\lvert \, \text{not entering a bush}]=\frac{1}{d}
\end{align} 
as new transition probabilities. This coincides with the probability of the first exit towards a neighbour $w_i$ of type $\tg$ starting in $v$. 
The new model actually lives on a new state space $\widehat{\Tr}$ that arises from $\Tr$ by erasing all {bushes},  see also Figure \ref{figurebush}. 
\begin{figure}[h]
 \centering
 \begin{minipage}[t]{0.48\textwidth}
\includegraphics[width=1\textwidth]{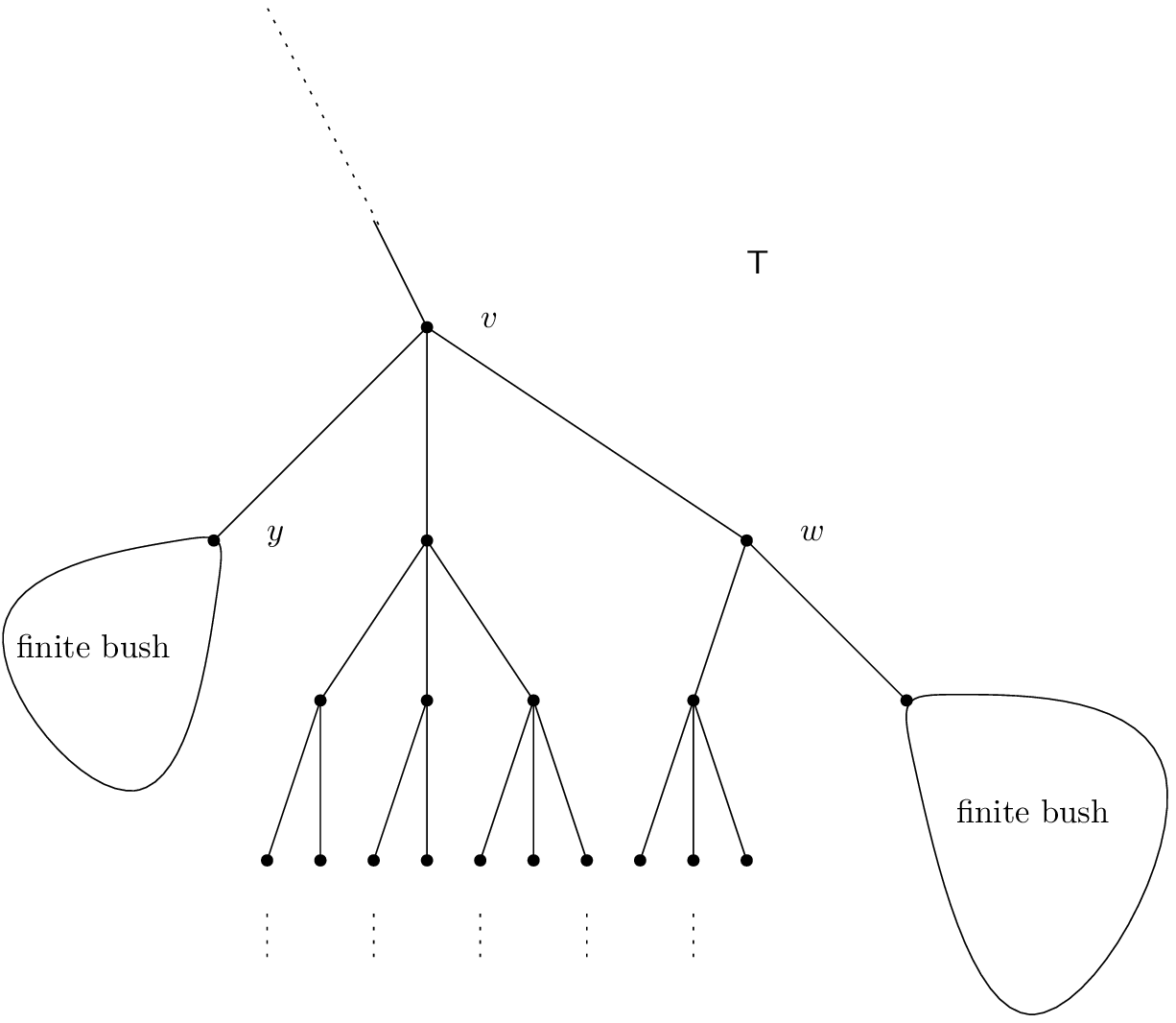}
\end{minipage}
\hspace{1mm}
\begin{minipage}[t]{0.48\textwidth}
\includegraphics[width=1\textwidth]{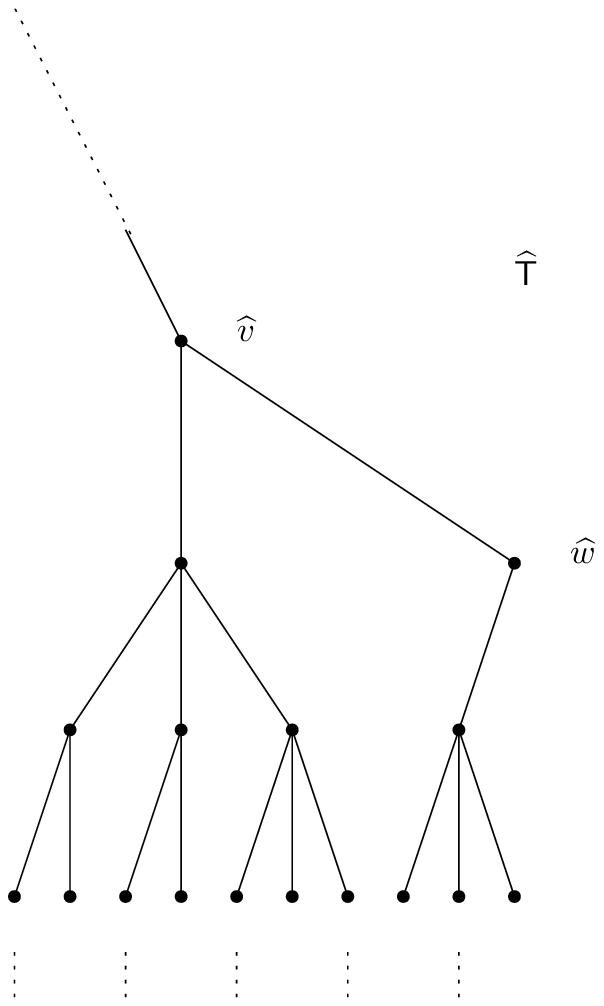}
\end{minipage}
\caption{Construction of $\widehat{\Tr}$ from $\Tr$ by deleting all bushes.}
 \label{figurebush}
\end{figure}
Then, we identify the frog configuration by $\eta^{*}(\widehat{v})=\eta^{*}(v)$ of the two models on $\Tr $ and $\widehat{\Tr}$ and obtain the new frog model $\FM( (\eta^{*}(\widehat{v}))_{\widehat{v}\in \widehat{\Tr}},\widehat{\Tr})$. We keep here the whole sequence of random variables in the frog configuration to point out that the random variables are not identically distributed. 

We denote by  $\nu$  the number of visits to the root in $\FM$  and by $\widehat{\nu}$ the number of visits to the root in $\FM( (\eta^{*}(\widehat{v}))_{\widehat{v}\in \widehat{\Tr}},\widehat{\Tr})$. Coupling the frog configuration at each vertex as in Lemma \ref{couplingstandard} we find for each frog in $\FM$ a corresponding frog in $\FM( (\eta^{*}(\widehat{v}))_{\widehat{v}\in \widehat{\Tr}},\widehat{\Tr})$. Using a coupling as in the proof of Lemma \ref{couplingstandard}  we obtain that  
\begin{equation}
\widehat{\nu}\succeq \nu. \label{BMCdominating}
\end{equation}
Thus transience of $\FM( (\eta^{*}(\widehat{v}))_{\widehat{v}\in \widehat{\Tr}},\widehat{\Tr})$ implies transience of $\FM$. 


\subsubsection*{Construction of a dominating BMC}
\begin{figure}[h]
 \centering
 \begin{minipage}[t]{0.29\textwidth}
\includegraphics[width=1\textwidth]{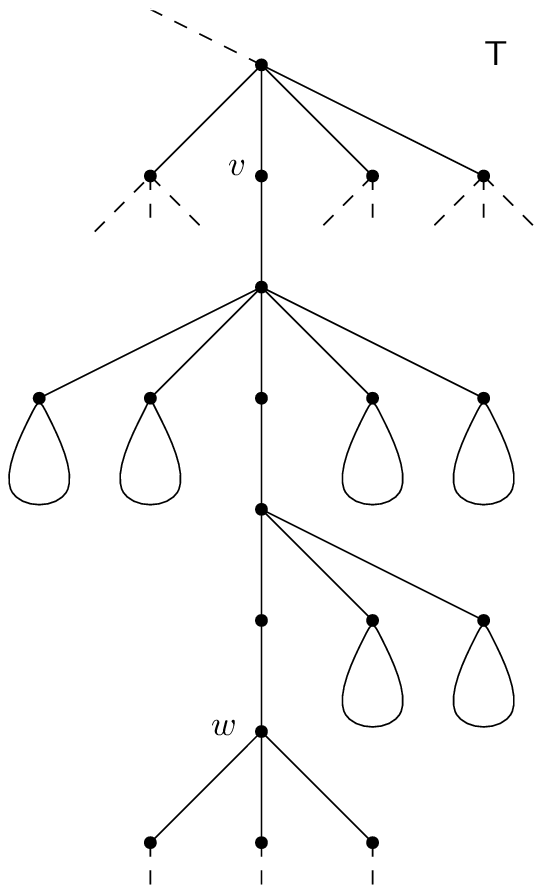}
\end{minipage}
\hspace{18mm}
\begin{minipage}[t]{0.29\textwidth}
\includegraphics[width=1\textwidth]{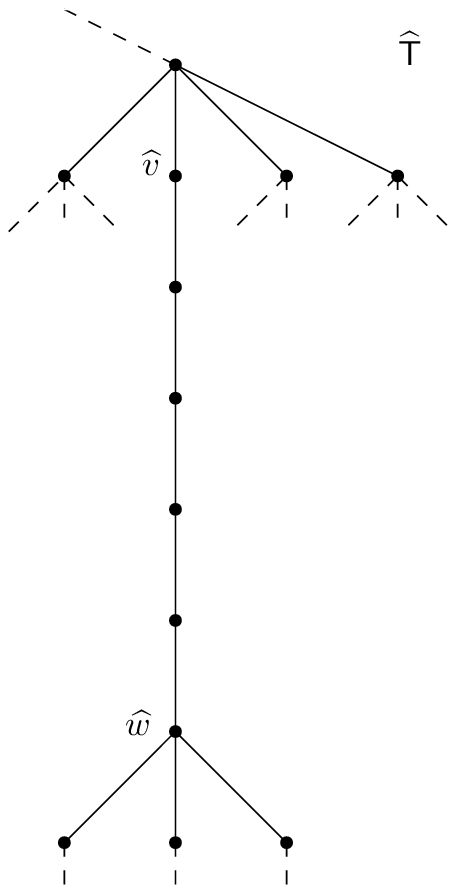}
\end{minipage}
\caption{Creating a stretch during the modification from $\Tr$ to $\widehat{\Tr}$.}
 \label{figure:stretch+bush_specialcase}
\end{figure}
Removing all bushes, we have to be aware that a sequence of vertices with only one child of type $\tg$ will create new stretches, see Figure \ref{figure:stretch+bush_specialcase}. Hence, we need to go on by using Proposition \ref{propositionstretches}. But the newly appeared stretches can be unbalanced  in the sense that some vertices were former neighbours to bushes and have the corresponding offspring and some not. This would inhibit the number of frogs emerging to the ends of the stretch to be equally distributed. Therefore, we modify the frog model in the following way: a vertex can have an offspring of at most $K$.  Therefore, every vertex which is part of a stretch could have at most $K-1$  finite bushes attached. We set $\widehat{\eta}({\widehat{v}}):=\sum_{i=1}^{K-1} \sum_{\widehat{w} \in G_{\widehat{v}_i}} \eta({\widehat{w}}) $ with $G_{\widehat{v}_i}$ being finite bushes generated according to $\T^{sub}$ for each vertex $\widehat{v}\in \widehat{\Tr}$ and notice that $(\widehat{\eta}({\widehat{v}}))_{\widehat{v}\in \widehat{\Tr}}$  is a sequence of i.i.d.\ random variables and we call their common measure $\widehat{\eta}$. Then, the model $\FM( (\eta^{*}(\widehat{v}))_{\widehat{v}\in \widehat{\Tr}},\widehat{\Tr})$ is dominated by $\widehat{\FM}:=\FM(\widehat{\Tr},\widehat{\eta})$, as there are only more particles in the new model and we can couple the two processes such that every visit in $\FM( (\eta^{*}(\widehat{v}))_{\widehat{v}\in \widehat{\Tr}},\widehat{\Tr})$  has a corresponding visit in $\widehat{\FM}$. 
\begin{figure}[h]
 \centering
\begin{minipage}[t]{0.3\textwidth}
\vspace{0mm}
\includegraphics[width=1\textwidth]{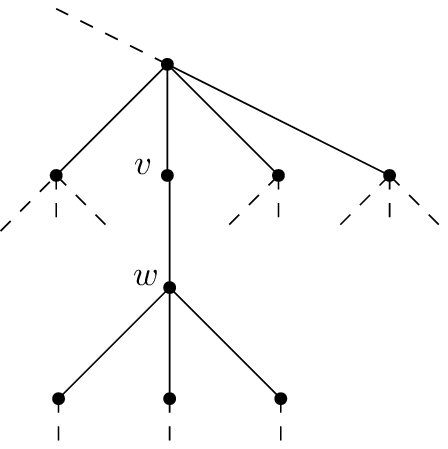}
\end{minipage}
\caption{The modification of the stretch in Figure \ref{figure:stretch+bush_specialcase} in the step from $\widehat{\Tr}$ to $\Tr'$.}
\label{figure:stretch+bush_specialcase_filledup}
\end{figure}

In the same manner as in Proposition \ref{propositionstretches} we want to couple $\widehat{\FM}$ with a modified model $\FM'$ doing the same steps as in Proposition \ref{propositionstretches}: if a frog enters a  stretch all frogs from the stretch are woken up and placed according to their exit measures at the two ends of the stretch. This results in the modified state space $\Tr'$ by merging the stretches into one vertex like in Proposition \ref{propositionstretches}, see Figure \ref{figure:stretch+bush_specialcase_filledup}. As there are $\sum_{i=1}^{K-1} \sum_{\widehat{w} \in G_{\widehat{v}_i}} \eta({\widehat{w}}) $ frogs placed on each vertex, from a stretch of length $\ell_{\widehat{v}}$ leave on average
\begin{align*}
\E^{\Tr'}_{v'}\left[F_{u'}\right] &= \frac{\ell_{\widehat{v}} \E[G](K-1)  \bar{\eta}}{2} + \frac{\ell_{\widehat{v}}}{\ell_{\widehat{v}} +1}\, , \, \\ 
\E^{\Tr'}_{v'}\left[{F}_{w'}\right] &= \frac{\ell_{\widehat{v}}(K-1) \E[G] \bar{\eta}}{2} + \frac{1}{\ell_{\widehat{v}}+1}\,
\end{align*}
 frogs to the two ends of the stretch. Here, the length $\ell$ of the stretch is distributed according to $geo(\widehat{p}_1)+1$, where  $\widehat{p}_1$ is the probability of having only one child of type $\tg$. 
 For the construction of a dominating $\BMC$ let again $N\in \N$ and  define the tree $\widehat{\Tr}_{N}$ as a copy of $\widehat{\Tr}$, where each stretch of length larger than $N$ is replaced by a  stretch of length $N$. On this tree we define again $\widehat{\BMC}_{N}$, on $\widehat{\Tr}_{N}$, with driving measure $\SRW$ and the offspring distribution $\mu$ is equal to the distribution which fulfills 
 \[ \mu_k(\widehat{v})=\Prob \left[ \sum_{i=1}^{K-1} \sum_{\widehat{w} \in G_{\widehat{y}_i}} \eta({\widehat{w}}) +1=k\right]  \] 
 for any $\widehat{v}\in \widehat{\Tr_N}$. Its mean offspring is denoted by $\bar{\mu}$.  We recall that $\Tr'$ is the tree, where the stretches of maximum length $N$ are compressed to a single vertex (similar to Proposition \ref{propositionstretches}). Then $\widehat{\BMC}_{N}$ defines naturally a $\BMC'_N=\widehat{\BMC}'_{N}$ on ${\Tr}'$: Once a particle enters a former stretch,  it produces offspring particles according to the exit-measures.

To find an $N \in \N$ such that ${\BMC}'$ is dominating for ${\FM}'$ we proceed like in the proof of Proposition \ref{propositionstretches} with the difference that in average to both sides of a geodesic stretch of length $\ell$ exit 
\[ \frac{\ell \E[G](K-1)  \bar{\eta}}{2}\]
frogs instead of $\frac{\bar{\eta}\ell}{2}$ frogs. The frog which is waking up the stretch leaves the stretch to each side with the same probability as before. Moreover the length of the stretch is now distributed according to $geo(\widehat{p}_1)+1$ and the probability that a vertex is dedicated as a starting vertex of a stretch is $ber(\widehat{p}_1)$-distributed, as well. 

The ${\BMC}'$ has to fulfill the transience criterion Theorem \ref{transiencecriterium}, as well. We notice, that $\widehat{\Tr}_N$ corresponds to the tree $\Tr_N$ from the construction of the dominating Branching Markov chain in the proof of Proposition \ref{propositionstretches} and 
\[\rho(\widehat{\Tr}_N)= \left(\cos \frac{\arccos \left(\frac{{2}\sqrt{d_{min}}}{d_{min}+1}\right)}{N+1}\right)\]
with $d_{min}=\min\{k\geq 2: p_k>0\}$. All together, using the same line of arguments as in Proposition \ref{propositionstretches}, we have the following conditions on $\eta$ and $\bar{\mu}$ such that there exists a dominating $\BMC'_N$:
\begin{enumerate} 
\item $1+\E[G]\bar{\eta} (K-1) < \bar{\mu}$;
\item $\frac{d_{min}+1}{d_{min}}\leq \bar{\mu} $ where $d_{min}=\min\{k \geq 2: \widehat{p}_k>0 \}$;
\item Choosing $\eta$ such that $\bar{\eta}$ is small enough such that there exists an $N$ such that 
\begin{align}
&\left( \frac{\bar{\eta}\, \E[G] (K-1)}{2}\left(\frac{1}{1-\widehat{p}_1}+1\right)+ \sum_{\ell=1}^{\infty} \left( \frac{\ell}{\ell+1} \right) \widehat{p}_1^{\ell-1} (1-\widehat{p}_1) \right)   \\
 & \quad < \sum_{\ell=1}^{N-1} g(\ell+1) \widehat{p}_1^{\ell-1} (1-\widehat{p}_1) + g(N) \widehat{p}_1^{N},
\end{align}
and (\ref{eq:comparisontype3}) hold;
\item Choosing $\eta$ such that for given $\widehat{p}_1$ and the previously selected ${N}$ equation
\begin{align}
&  \bar{\eta} \, \E[G] (K-1) \left(\frac{3! (N+1)^2}{(\arccos(1/ \bar\mu)^2}\right) \left(  \sum_{\ell=1}^{N-1} \frac{\widehat{p}_1^{\ell+1}}{(l+1)^{2k-1}}+\left(\frac{1}{N+1}\right) \frac{\widehat{p}_1^{N-1}}{(1-\widehat{p}_1)}\right)   \\
&\quad\leq  \sum_{l=1}^{N-1} \left(\frac{\widehat{p}_1^{l+1}}{(l+1)^{2k-1}}\right)   
\end{align} holds;
\item $\bar{\mu} <   \left(\cos \frac{\arccos \left(\frac{{2}\sqrt{d_{min}}}{d_{min}+1}\right)}{N+1}\right)^{-1}$.
\end{enumerate}

We can conclude similar to Proposition \ref{propositionstretches}. 

\end{proof}

\section[Some properties]{Some properties of Galton--Watson trees and branching random walks} \label{appendix_frogs}
\subsection{The relation with generating functions}\label{sec:A1}
 At various places we have used generating functions. They are a crucial tool in the study of BMC, e.g.,~see \cite{BePe:94}, \cite{Candellero},  \cite{Hueter2000}, \cite{MV}, and  \cite{woess}.
Let $M$ be a subset of the state space and modify the BMC in a way such that particles are absorbed in $M$ and once they have arrived in $M$, they keep on producing one offspring a.s. In other words, particles arriving in $M$ are \emph{frozen}. Set $Z_{\infty}(M)\in \N\cup\{\infty\}$ as the total number of frozen particles in $M$ at time ``$\infty$''.  For $M\subseteq \Gamma$, we define the first visiting generating function:
$$
F(x,M|z)  :=  \sum_{n\geq 0} \Prob\bigl[Z_n\in M,\forall m\leq n-1:
Z_m\notin M \mid X_0=x\bigr]z^n,
$$
 where $Z_{n}$ is the original SRW and $\Prob$ its corresponding probability measure.
 The following lemma will be used several times in our proofs; a short proof can be found for example in \cite[Lemma 4.2]{Candellero}.
 
 \begin{lemma}\label{lemma:Z-infty-formel}
{Let $\bar \mu$ be the mean offspring of the BMC.} For any $M\subseteq \Gamma$, we have \[\mathbb{E}\bigl[Z_{\infty}(M)\bigr]=F(e,M|\bar \mu)\, .\] 
\end{lemma}

\subsection{Spectral radius of trees} \label{subsec:spectralradiusoftrees}
In order to study recurrence and transience of a BMC it is essential to understand the spectral radius of the underlying Markov chain. In this section, we collect several results on the spectral radius of SRW on trees.

\begin{definition}\label{def:isoperimetricconstant}
The {isoperimetric constant} $\iota(T)$ of a tree with edges $E$ and vertices $V$ is defined by
\[\iota(T):=\inf\left\{\frac{\mid \delta_E F\mid}{Vol(F)}:F\subset X \, \mbox{finite} \right\}\]
where $\delta_E F=E(F,X\setminus F)$ is the set of edges connecting $F$ with $T\setminus F$ and $Vol(F)=\sum_{x\in F} deg(x)$.
\end{definition}
For the isoperimetric constant it holds that, $\iota(T) =0 $ if and only if the spectral radius $\rho(T)$ of the simple random walk  equal to $1$, see  Theorem 10.3 in \cite{woessRW}. 

There is a more precise statement on finite approximation of the spectral radius, e.g.,~see \cite{BePe:94a} and \cite{mullerrecurrenceBMC}.  Consider  an infinite irreducible Markov chain $(X,P)$ and write $\rho(P)$ for its spectral radius.  A
subset $Y\subset X$ is called irreducible if the sub-stochastic
operator $$P_Y=(p_Y(x,y))_{x,y\in Y}$$ defined by
$p_Y(x,y):=p(x,y)$ for all $x,y\in Y$ is irreducible. It is rather straightforward to show the next characterization.

\begin{lemma}\label{lem:specapprox}
Let $(X, P)$ be an irreducible Markov chain. Then, \begin{equation}\label{eq:specapprox}
\rho(P)=\sup_Y \rho(P_Y),
\end{equation}
where the supremum is over finite and irreducible subsets
$Y\subset X.$ Furthermore, $\rho(P_F)< \rho(P_G)$ if $F\subsetneq
G.$
\end{lemma}

We compare this also to the Perron-Frobenius theorem, see for example \cite{Seneta}, especially for the last inequality. A first observation is the following result, see \cite[Lemma 9.86]{woess}. We say that  a stretch (or unbranched path) of length $N$ in a tree $T$ is a path $[v_{0}, v_{1}, \ldots, v_{N}]$ of distinct vertices such that $deg(v_{k})=2$ for $k=1,\ldots, N-1$.

\begin{lemma}\label{lem:unbranched}
Let $T$ be a locally finite tree $T$. If $T$ contains stretches of arbitrary length, then $\rho(T)=1$.
\end{lemma}

Moreover,  we can give a precise characterization of the spectral radius of a simple random walk on a GW-tree
\begin{lemma} \label{lemmaspectralradii}
Let $\rho(\Tr)$ be  the spectral radius of the simple random walk on  a Galton--Watson tree $\Tr$ with offspring distribution $(p_{i})_{i\geq 0}$. Then,
\begin{itemize}
\item if $p_0+p_1>0$ we have  $\rho(\Tr)=1$  for $\GW$-a.a.~realizations $\Tr$;\item if $p_0+p_1=0$ we have  $\rho(\Tr) = \rho(T_{d+1})=\frac{2\sqrt{d}}{d+1}<1$  for $\GW$-a.a.~infinite realizations  $\Tr$,
\end{itemize}
where $d=\min\{i:p_i> 0\}$ and $T_{d+1}$ is the homogeneous tree with offspring $d$.
\end{lemma}
\begin{proof}
If $\Tr$ is finite, the simple random walk is recurrent and it holds that $\rho(\Tr)=1$, see Section 1 in \cite{woessRW}.
Now, let  us assume that $\Tr$ is  infinite. In the case where $p_1>0$ the tree contains, for every choice of $N\in\N$, $\GW$-a.s.~a stretch of length $N$; this is a consequence of the lemma of Borel--Cantelli.  Using Lemma \ref{lem:unbranched} we conclude that $\rho(\Tr)=1$.
Now, we assume that $p_1=0$ but $p_0>0$. In this case the tree $\Tr$ contains, for every choice of $N\in\N$, a  finite bush of $N$ generations, which we call bush $B_{N}$.  For such a bush $B_N$ it holds that  $\frac{\mid \delta_E B_N\mid}{Vol(B_N)}\leq\frac{1}{2N}$. Again, by finding an arbitrary large bush we obtain $\iota(\Tr)=0$ and consequently using Theorem 10.3 in \cite{woessRW} we conclude $\rho(\Tr)=1$.
 In the case $p_0+p_1=0$  Corollary 9.85 in \cite{woess} implies that $\rho(\Tr)\leq\rho(T_{d+1})=\frac{2\sqrt{d}}{d+1}$  where $d$ is the smallest  offspring of the Galton--Watson tree and $T_{d+1}$ denotes the homogeneous tree with offspring $d$. The remaining equality follows  by finding arbitrarily  large balls of $T_{d+1}$ as copies in $\Tr$ as above  and applying Lemma  \ref{lem:specapprox}.
\end{proof}

We construct a new tree $\widetilde T$ by replacing each edge $e$ of $T$ with a stretch  of length $k=k(e)$. We call $\widetilde T$ a subdivision of $T$ and $\max_{e}\{k(e)\}$ the maximal subdivision length of $\widetilde T$. We write $T_{(N)}$ for the subdivision of $T$ where $k(e)=N$ for all edges $e$ in $T$. 
We state a particular case of Theorem 9.89 in \cite{woess}. 
\begin{theorem}\label{thm:rhoN}
Let $T$ be a locally finite tree and denote $\rho(T)$ (resp.\ $\rho(T_{(N)})$) the spectral radius of the SRW on $T$ (resp.\ $T_{(N)}$). Then, 
\begin{enumerate}
\item[a)] \begin{equation} \rho(T_{(N)})= \cos \frac{\arccos \rho(T)}{N};
\end{equation}
\item[b)] if $\widetilde T$ is an arbitrary subdivision of $T$ of maximal subdivision length $N$ then
\begin{equation}
\rho(T) \leq \rho(\widetilde T) \leq \rho(T_{(N)}).
\end{equation}
\end{enumerate}
\end{theorem}

\subsection{Absorbing BMC on finite paths}\label{sec:absBMC}
We consider the SRW, $(Z_{n})_{n\geq 0},$ on an unbranched path of length $N$ with absorbing states $v_{0}$ and $v_{N}$. In other words, we consider the ruin problem (or birth-death chain) on $[N]:=\{0,1,\ldots, N\}$ defined through the transition kernel $P_{N}=( p_{N}(x,y))_{x,y\in [N]}$: $p_{N}(0,0)=p_{N}(0,N)=1$ and $p_{N}(x,x+1)=p_{N}(x,x-1)=1/2$ for $1 \leq x \leq N-1$. We set $\rho(P_{N})$ for the spectral radius of the reducible class $\{1,\ldots, N-1\}$. Let 
\begin{equation}
f^{(n)}_{N}(x,y) := \Prob[Z_{n}=y, Z_{k}\neq y~\forall~ 0\leq k<n | Z_{0}=x]
\end{equation}
and define the first visit generating function
\begin{equation}
F_{N}(x,y |z):= \sum_{n=0}^{\infty} f^{(n)}_{N}(x,y) z^{n}. \label{def:generatingfunc}
\end{equation}
The convergence radius of the power series equals $R_{N}=1/\rho(P_{N})$.

We give the following expressions of the generating function $F_{N}$ for two particular pairs of values of $x$ and $y$, see Example 5.6 in \cite{woess}.

\begin{lemma}\label{lem:F}
Let $1\leq z \leq R_{N}$ and $\varphi$ such that $1/z=\cos \varphi$. Then,
\begin{equation}
F_{N}\left(1,N \middle\vert  \frac{1}{\cos\varphi}\right)= \frac{\sin  \varphi}{\sin N \varphi} \mbox{ and } F_{N}\left(N-1,N \middle\vert  \frac{1}{\cos\varphi}\right)= \frac{\sin (N-1) \varphi}{\sin N \varphi}.
\end{equation}
\end{lemma}

We present lower bounds of these generating functions; the index shift is done to improve the presentation of the proofs in the main part.

\begin{lemma}\label{lem:Fapprox}
Let $1\leq z \leq R_{N}$ and $\varphi$ such that $1/z=\cos \varphi$. Then, \begin{equation}
F_{N+1}\left( N, N+1 \middle \vert  \frac{1}{\cos\varphi}\right) \geq \frac{N}{N+1} \left(1+ \frac{ \left(1+2N\right)\varphi^{2}}{3!}\right),
\end{equation}
\begin{equation}
F_{N+1}\left( 1, N+1 \middle \vert  \frac{1}{\cos\varphi}\right) \geq \frac{1}{N+1} \left(1+ \frac{ \left(2N+N^2\right)\varphi^{2}}{3!}\right).
\end{equation}
\end{lemma}
\begin{proof}
For proving the above approximations we will use the infinite product expansion \[\sin(z)=z \prod_{n=1}^{\infty} \left(1-\frac{z^2}{n^2\pi^2}\right), \, z\in \mathbb{C}\] and the power series expansion  
\[\sin(z)=\sum_{n=0}^{\infty} \frac{(-1)^{n} z^{2n+1}}{(2n+1)!},\, z\in \mathbb{C}\]
of the sine. Now, 
\begin{align}
F_{N}\left( N, N+1 \middle \vert  \frac{1}{\cos\varphi}\right) & =\frac{\sin (N) \varphi}{\sin (N+1) \varphi}= \frac{N\varphi}{(N+1) \varphi}\frac{\prod_{n=1}^{\infty} \left(1-\frac{(N\varphi)^2}{n^2 \pi^2}\right)}{\prod_{n=1}^{\infty} \left(1-\frac{((N+1)\varphi)^2}{n^2 \pi^2}\right)}\\
&=\frac{N}{N+1}\frac{\prod_{n=1}^{\infty} \left(1-\frac{N^2\varphi^2+2N\varphi^2+\varphi^2}{n^2 \pi^2}+\frac{2N\varphi^2+\varphi^2}{n^2\pi^2}\right)}{\prod_{n=1}^{\infty} \left(1-\frac{N^2\varphi^2+2N\varphi^2+\varphi^2}{n^2 \pi^2}\right)}\\
&\geq \frac{N}{N+1}\left( \prod_{n=1}^{\infty} \left(1+\frac{2N\varphi^2+\varphi^2}{n^2\pi^2}\right) \right)\, .
\end{align}
Defining $z=i\varphi\sqrt{1+2N}$ we obtain by using the product expansion and afterwards the power series expansion, that
\begin{align*}
\frac{N}{N+1}\left( \prod_{n=1}^{\infty} \left(1+\frac{2N\varphi^2+\varphi^2}{n^2\pi^2}\right) \right) &= \frac{N}{N+1}\frac{\sin(z)}{z} \\
&\geq \frac{N}{N+1}\left(1+\frac{\varphi^2(1+2N)}{3!}\right)\, .
\end{align*}
The second part follows the exact same line as the first part of the proof. 
\end{proof}

 {\textbf{Acknowledgement:} We want to express our great appreciation to Wolfgang Woess for many helpful discussions. Our thanks are extended to Nina Gantert for advice and discussions. We would also like to thank Stefan Lendl for his generous support in all numerical issues and Ecaterina Sava-Huss for raising the idea for this project.  We are grateful to Marcus Michelen and Josh Rosenberg for pointing out a mistake in the treatment of ``bushes, no stretches'' in a previous version. We also want to thank the anonymous referees whose comments lead to a considerable improvement of this paper. 

The second author was supported by the Austrian Science Fund (FWF): W1230. Grateful acknowledgment is made for hospitality from the Institute of Discrete Mathematics of TU Graz, where the research was carried out during the first author's visits.  }
\appendix

\bibliographystyle{abbrv}

\bibliography{literatur}

\bigskip
\noindent \begin{minipage}{0.48\textwidth}
Sebastian M\"{u}ller\newline
Aix Marseille Universit\'e\newline
CNRS,  Centrale Marseille\newline
I2M\newline
UMR 7373\newline 13453 Marseille, France\newline
\texttt{sebastian.muller@univ-amu.fr}
\end{minipage}
\hfill
\begin{minipage}{0.48\textwidth}
Gundelinde Maria Wiegel \newline
Institute of Discrete Mathematics,\newline
Graz University of Technology\newline
Steyrergasse 30,\newline
8010 Graz, Austria\newline
\texttt{wiegel@math.tugraz.at}
\end{minipage}

\end{document}